\documentclass[reqno]{amsart}

\usepackage[latin1]{inputenc}
\usepackage[english]{babel}
\usepackage{enumerate}
\usepackage{graphicx}

\usepackage[colorlinks=false,backref=false,pagebackref=false,pdftex,pdfauthor={Ederson Moreira dos Santos, Filomena Pacella},pdftitle={Some concentration profiles in elliptic equations}]{hyperref}

\usepackage{amsthm}
\usepackage{amsmath,amsfonts,amssymb}

\usepackage{latexsym,hyperref}
\usepackage{comment}

\usepackage[showonlyrefs]{mathtools}
\mathtoolsset{showonlyrefs=true}

\newcommand{\R}{{\mathbb R}}
\newcommand{\Z}{{\mathbb Z}}

\newcommand {\lbk}{\linebreak}

\newcommand {\menos}{\backslash}
\newcommand {\gd}{\displaystyle}

\newcommand {\con}{\subset}

\newcommand {\rt}{\rightarrow}

\newcommand {\fraco}{\rightharpoonup}

\numberwithin{equation}{section}

\newtheorem{theorem}{Theorem}[section]
\newtheorem{corollary}[theorem]{Corollary}
\newtheorem{lemma}[theorem]{Lemma}
\newtheorem{proposition}[theorem]{Proposition}

\theoremstyle{definition}
\newtheorem{definition}[theorem]{Definition}

\newtheorem{remark}[theorem]{Remark}

\begin{document}

\title{Hénon type equations and concentration on spheres}

\thanks{E. Moreira dos Santos is partially supported by CNPq and FAPESP. F. Pacella is partially supported by PRIN 2009-WRJ3W7 grant}

\author{Ederson Moreira dos Santos}
\address{Ederson Moreira dos Santos \newline \indent Instituto de Ciências Matemáticas e de Computação \newline \indent Universidade de São Paulo \newline \indent
Caixa Postal 668, CEP 13560-970 - S\~ao Carlos - SP - Brazil}
\email{ederson@icmc.usp.br}

\author{Filomena Pacella}
\address{Filomena Pacella \newline \indent Dipartimento di Matematica \newline \indent Università di Roma {\it{Sapienza}} \newline \indent
P.le. Aldo Moro 2, 00184 Rome, Italy}
\email{pacella@mat.uniroma1.it}

\date{\today}

\subjclass[2000]{35B06; 35B07;  35B40;  35J15; 35J61}
\keywords{Hénon type problems; Black holes; Concentration phenomena; Symmetry}

\begin{abstract}
In this paper we study the concentration profile of various kind of symmetric solutions of some semilinear elliptic problems arising in astrophysics and in diffusion phenomena. Using a reduction method we prove that doubly symmetric positive solutions in a $2m$-dimensional ball must concentrate and blow up on $(m-1)$-spheres as the concentration parameter tends to infinity. We also consider axially symmetric positive solutions in a ball in $\R^N$, $N \geq 3$, and show that concentration and blow up occur on two antipodal points, as the concentration parameter tends to infinity. 
\end{abstract}
\maketitle

\section{Introduction}
We consider semilinear elliptic problems of the type
\begin{equation}\label{general equation}
\left\{
\begin{array}{l}
- \Delta u = h(x)|u|^{p-2}u \quad \text{in} \quad B_N(0,1),\\
u = 0 \quad \text{on} \quad \partial B_N(0,1),
\end{array}
\right. 
\end{equation}
where $B_{N}(0,1)$ is the unit open ball centered at the origin in $\R^N$, $N \geq 3$, and $p>2$.

If $N = 2m$, $m > 1$ and $x = (y_1, y_2)$, $y_i \in \R^m$, $i=1,2$, we take
\begin{equation}\label{weight henon}
 h(x) = |x|^{\alpha} = |(y_1, y_2)|^{\alpha}, \quad \alpha>0
\end{equation}
or
\begin{equation}\label{weight 1.3}
 h(x) = |y_2|^{\alpha}, \quad \alpha>0.
\end{equation}

If $N\geq 3$ and $x = (x_1, \ldots, x_N) \in \R^N$ we take
\begin{equation}\label{weight 1.4}
h(x) = |x_N|^{\alpha}, \quad \alpha>0.
\end{equation}

The first choice corresponds to the case of the well-known Hénon equation \cite{henon}, while the other two are variants which have interest in applications as we will explain later.

From the mathematical point of view problem \eqref{general equation} has an interesting and rich structure and various results have been proved so far, some of which will be described in the sequel. Let us recall that, in the case of \eqref{weight henon}, it was first observed in \cite{ni} that the presence of the weight $|x|^{\alpha}$ modifies the consequences of the Poho\v zahev identity and produces a new critical exponent, namely $2(N+\alpha)/(N-2)$, for the existence of classical solutions. In \cite{BW,cao-peng-yan,smets-su-willem}, symmetry breaking, asymptotics and single point concentration profile at the boundary of the least energy solutions, as $\alpha \rt \infty$, are described.  Moreover, in \cite{pacella,smets-willem},  it is proved that any least energy solution is foliated Schwarz symmetric. More recently, the existence of infinitely many positive solutions have been proved in \cite{wei-yan}, in the case $p = 2N/(N-2)$.  

The main purpose of this paper is to present a new  feature of the Hénon equation and of its variant with $h(x)$ given by \eqref{weight 1.3}, namely the existence of positive solutions concentrating on spheres, as $\alpha \rt \infty$.

In the study of semilinear elliptic equations with power-like nonlinearities there are few results of this type, unlike the case the concentration at a single point about which a large literature is available. For the case of \eqref{weight 1.4} we will show instead concentration at antipodal points for some axially symmetric solutions, as $\alpha \rt \infty$.

Before stating precisely our results let us outline the connections between our problems and some mathematical models arising in astrophysics and diffusion processes. 

It was during the golden age of general relativity, from 1960 to 1975, that astrophysicists started to devote intensive attention to understand and to detect the existence of black holes in globular clusters. In 1972, Peebles \cite{peebles1,peebles2} published seminal works describing a stationary distribution of stars near a massive collapsed object, such as a black hole, located at the center of a globular cluster.  In recent years, as in \cite{existBHGC}, the existence of single black holes in globular clusters have been perceived. More recently, it was reported in \cite{twoblackholes} the presence of two flat-spectrum radio sources in the Milky Way globular cluster M22.

It can be derived that the existence of stationary stellar dynamics models, cf. \cite{batt,batt-f-horst,li,li-santanilla,batt-li}, is equivalent to the solvability of the equation
\[
-\Delta U(x)  = f(|x|, U(x)) \quad \text{in} \quad \R^3,
\]
which, in the case $f(|x|, U) = |x|^{\alpha}|U|^{p-2}U$, $\alpha>0$, $p>2$, becomes the Hénon equation,
\[
-\Delta U(x)  = |x|^{\alpha}|U|^{p-2}U \quad \text{in} \quad \R^3,
\]
where the weight $|x|^{\alpha}$ represents a black hole located at the center of the cluster, whose absorption strength increases as $\alpha$ increases. This corresponds to the case of problem \eqref{general equation} with $h(x)$ given by \eqref{weight henon}, while when we take $h(x)$ as in \eqref{weight 1.3} or \eqref{weight 1.4} it corresponds to a supermassive absorbing object represented by a higher dimensional body.

Besides its application to astrophysics, problem \eqref{general equation} also models steady-state distributions in other diffusion processes.  For example, suppose $u(x)$ represents the density of some chemical solute, as a function of the position $x$, confined in a ball $B$. In this case $|u|^{p-2}u$ corresponds to the reaction term and the weight $h(x)$ is an intrinsic property of the medium $B$, which inhibits diffusion at the region $A$ where $h=0$ and hampers diffusion close to $A$. So, in case $h$ is one of \eqref{weight henon}, \eqref{weight 1.3} or \eqref{weight 1.4},  stronger obstructions correspond to larger values of $\alpha$. This line of reasoning leads that, as $\alpha \rt \infty$, concentration on parts of the domain far away from $A$ should occurs; cf. \cite{BW,cao-peng-yan} in the case of \eqref{weight henon}. 

Our results show that models having objects which inhibit diffusion, \linebreak represented by a point as in \eqref{weight henon} or by objects of higher dimension as in \eqref{weight 1.3}, can produce concentration on spheres, located as far as possible from the absorbing object as the parameter $\alpha$ tends to infinity.

We point out that our results, in the cases of \eqref{weight henon} and \eqref{weight 1.3} holds in balls contained in $\R^{2m}$, $m>1$ and hence not in $\R^3$ which is the relevant case for the astrophysics models. We believe that the concentration phenomenon  we describe should give insights also for the $3$-dimensional case. However, for other diffusion processes it is meaningful to pose the problem in higher dimensional spaces. 

The result we get regarding problem \eqref{general equation} with $h(x)$ given by \eqref{weight 1.4} instead applies to any dimension $N \geq 3$, so, in particular, covers the case of the astrophysical model.

To state precisely our results we need to introduce some notations that we will keep throughout the paper.  For each $r >0$, we set
\begin{equation}\label{sphere radius r in rk}
\begin{array}{l}
 B_k(0,r) := \{y \in \R^{k}; |y| < r\}, \quad S^{k-1}_r : = \{ y \in \R^{k}; |y| = r \} \quad \text{and}\\\\
 S^{k-1} : = \{ y \in \R^{k}; |y| = 1 \}.
 \end{array}
 \end{equation}
In addition, for sake of clarity, since we will perform changes of variables that will affect the space dimension, for any function $u: \Omega \con \R^k \rt \R$ we denote the usual Laplacian of $u$ in $\R^k$ by
\[
\Delta_{k}u (x) = \sum_{i =1}^{k} u_{x_ix_i}(x), \quad \forall \, x \in \Omega.
\]

We will say that $u: B_{2m}(0,1) \rightarrow \R$, $m \geq 1$, is doubly symmetric if
\[
u(y_1, y_2) = u (|y_1|, |y_2|), \quad \forall \, (y_1, y_2) \in B_{2m}(0,1).
\] 

Our first result concerns the case \eqref{weight henon} and describes the concentration profile of doubly symmetric solutions of the Hénon equation, that is:
\begin{equation}\label{radial 2m henon introduction}
\left\{
\begin{array}{l}
- \Delta_{2m} u = |(y_1, y_2)|^{\alpha} |u|^{p-2}u, \quad (y_1, y_2) \in B_{2m}(0,1), \\
u = 0 \quad \text{on} \quad \partial B_{2m}(0,1).
\end{array}
\right.
\end{equation}

\begin{theorem} \label{theorem henon concentration}
Assume $m > 1$,  $2 < p < \frac{2(m+1)}{m-1}$.  Then there exists $\alpha_0 = \alpha_{0}(p,m) > 4$ such that for each $\alpha > \alpha_0$ the following holds. Let $u_{\alpha}$ be any least energy solution among the doubly symmetric solutions of \eqref{radial 2m henon introduction}. Then, up to replacing $u_{\alpha}$ by $-u_{\alpha}$ and up to commuting $y_1$ and $y_2$, one has $u_{\alpha} > 0$ in $B_{2m}(0,1)$, $u_{\alpha}$ is not radially symmetric and there exists $0 < r_{\alpha} < 1$ such that
 \[
\mathcal{M}_{\alpha}: = \max_{(y_1, y_2) \in B_{2m}(0,1)} u_{\alpha}(y_1, y_2) = u_{\alpha}(y_{1}, 0), \quad \forall \; y_1 \in S^{m-1}_{r_{\alpha}}.
 \]
Moreover, $u_{\alpha}$ concentrates and blows up on $S^{m-1} \times \{0 \} \con \R^{2m}$, i.e., $r_{\alpha} \rt 1$, $\mathcal{M}_{\alpha} \approx \alpha^{2/(p-2)}$ and $\alpha (1 - r_{\alpha}) \rt \ell$ for some positive number $l$, as $\alpha \rt \infty$. 
\end{theorem}

Next we consider the case \eqref{weight 1.3} and describe the concentration profile of doubly symmetric solutions of the equation
\begin{equation}\label{partial henon equation introduction}
 \left\{
\begin{array}{l}
- \Delta_{2m} u = |y_2|^{\alpha} |u|^{p-2}u, \quad (y_1, y_2) \in B_{2m}(0,1), \\
u = 0 \quad \text{on} \quad \partial B_{2m}(0,1).
\end{array}
\right.
\end{equation}

\begin{theorem} \label{theorem partial henon concentration}
Assume $m > 1$,  $2 < p < \frac{2(m+1)}{m-1}$. Then there exists $\alpha_0 = \alpha_{0}(p,m) > 4$ such that for each $\alpha > \alpha_0$ the following holds. Let $u_{\alpha}$ be any least energy solution among the doubly symmetric solutions of \eqref{partial henon equation introduction}. Then, up to replacing $u_{\alpha}$ by $-u_{\alpha}$, one has $u_{\alpha} > 0$ in $B_{2m}(0,1)$, $u_{\alpha}$ is not radially symmetric and there exists $0 < r_{\alpha} < 1$ such that
  \[
\mathfrak {M}_{\alpha}: = \max_{(y_1, y_2) \in B_{2m}(0,1)} u_{\alpha}(y_1, y_2) = u_{\alpha}(0, y_2), \quad y_2 \in S^{m-1}_{r_{\alpha}}.
 \]
Moreover, $u_{\alpha}$ concentrates and blows up on $\{0\} \times S^{m-1} \con \R^{2m}$, i.e., $r_{\alpha} \rt 1$, $\mathfrak {M}_{\alpha} \approx \alpha^{2/(p-2)}$ and $\alpha (1 - r_{\alpha}) \rt \ell$ for some positive number $l$, as $\alpha \rt \infty$. 
\end{theorem}

Let us stress that from the physical point of view, the concentration phenomenon described in Theorem \ref{theorem partial henon concentration} is indeed expected. Since the set
\[
\mathfrak{D_m} = \{ (y_1, 0) \in \R^m \times \R^m; |y_1| < 1 \}
\]
inhibits diffusion and its hindrance to diffusion increases with $\alpha$, the maximum point of the density $u$ is expected to be as far apart from $\mathfrak{D_m}$ as possible when $\alpha$ tends to infinity. However, the concentration profile described in \linebreak Theorem \ref{theorem henon concentration} is a bit less evident. Indeed in this case, the set $\mathfrak{D_0}$ which \linebreak inhibits diffusion is reduced to the origin $(0,0)$. Then there are three possible doubly symmetric sets, as far away as possible from $\mathfrak{D_0}$, where in principle, concentration could occurs, namely $S^{m-1} \times \{0 \} \con \R^{2m}$ (which up to rotation is the same as $\{0\} \times S^{m-1} \con \R^{2m}$), $S^{m-1} \times S^{m-1} \con \R^{2m}$ or $S^{2m-1} \con \R^{2m}$. Nevertheless, thinking about reducing the energy, the second and third \linebreak possible concentration profiles can be excluded since solutions concentrating on $S^{m-1} \times \{0 \} \con \R^{2m}$ have lower energy than those that concentrate on $S^{m-1} \times S^{m-1} \con \R^{2m}$ or on $S^{2m-1} \con \R^{2m}$. This is in fact the meaning of Theorem \ref{theorem henon concentration}.

Note that in Theorem \ref{theorem henon concentration} and Theorem \ref{theorem partial henon concentration} the exponent $2 < p < 2(m+1)/(m-1)$ can be larger than the critical Sobolev exponent in dimension $2N$, namely $4N/(2N-2)$, and that $2(m+1)/(m-1)$ is the critical Sobolev exponent in dimension $m+1$.

Finally we turn to the case \eqref{weight 1.4} where we are able to get results in any dimension $N \geq 3$. Let us point out that in the $3$-dimensional case with $x = (x_1, x_2, x_3) \in B_3(0,1)$, the weight $|x_3|^{\alpha}$ represents a two dimensional absorbing object described by 
\[
\mathfrak{D}_ 2= \{ (x_1,x_2, 0) \in \R^2 \times \R; |(x_1,x_2)| < 1 \},
\]
whose absorption strength increases as $\alpha$ increases. In this case, reasoning  as above, a concentration phenomenon on the antipodal points $(0,0,1)$ and $(0,0,-1)$ is expected as $\alpha \rt \infty$. Indeed, we prove this result in any dimension. 

\begin{theorem}\label{theorem two point concentration}
Let $N\geq3$ and $2 < p < 2N/(N-2)$. Consider the problem
\begin{equation}\label{hyper-plane of black holes}
\left\{
\begin{array}{l}
 -\Delta_{N} u = |x_N|^{\alpha} |u|^{p-2}u, \qquad x= (x_1, \ldots, x_N) \in B_{N}(0,1),\\
 u = 0 \quad \text{on} \quad \partial B_{N}(0,1).
\end{array}
\right.
\end{equation}
Let $u_{\alpha}$ be any least energy solution of \eqref{hyper-plane of black holes} among the solutions that are axially symmetric with respect to $\R e_{N} \con \R^{N}$ and symmetric with respect to $x_N$. Then, up to replacing $u_{\alpha}$ by $-u_{\alpha}$, one has $u_{\alpha} > 0$ in $B_{N}(0,1)$ and there exists $0 \leq r_{\alpha} < 1$ such that
  \[
\mathbf{M}_{\alpha}: = \! \! \max_{(x_1, \ldots, x_N) \in B_{N}(0,1)} u_{\alpha}(x_1, \ldots, x_N) \! = \! u_{\alpha}(0, \ldots, 0, r_{\alpha}) \! = \! u_{\alpha}(0, \ldots, 0, -r_{\alpha}).
 \]
Moreover, $u_{\alpha}$ concentrates and blows up, simultaneously, on the antipodal points $(0, \ldots, 0,1)$ and $(0, \ldots, 0, -1)$, i.e., $r_{\alpha} \rt 1$, $\mathbf{M}_{\alpha} \approx  \alpha^{2/(p-2)}$ and $\alpha (1 - r_{\alpha}) \rt \ell$ for some positive number $l$, as $\alpha \rt \infty$. 
\end{theorem} 

Let us explain shortly the strategies of the proofs of the above theorems.

To prove Theorem \ref{theorem henon concentration} and Theorem \ref{theorem partial henon concentration} we perform a change of variables which allows us to reduce problems \eqref{radial 2m henon introduction} and \eqref{partial henon equation introduction} to other semilinear problems in $\R^{m+1}$. The main feature of this transformation is that it sends in a bijective way doubly symmetric solutions of \eqref{radial 2m henon introduction} or \eqref{partial henon equation introduction} to axially symmetric solutions of the reduced problems. This explain why the exponent $p$ can be larger than the critical exponent in dimension $2N$. We mention that this approach was introduced in \cite{pacella-srikanth} to study concentration on spheres of some singularly perturbed problems in $\R^{2m}$ and relied on an idea of \cite{ruf-srikanth}.

So, in the case of \eqref{radial 2m henon introduction} and \eqref{partial henon equation introduction}, we are led to study the $(m+1)$-dimensional problems \eqref{radial m+1 henon} and \eqref{partial henon equation reduced} respectively, for which analyzing the behavior, as $\alpha \rt \infty$, of the solutions we get point concentration results, which, going back to the $2m$-dimensional problems, imply concentration on $(m-1)$-spheres.

We stress that while the reduced problem \eqref{radial m+1 henon} corresponding to \eqref{radial 2m henon introduction} is still a Hénon problem for whose analysis we can use existing results, the reduced problem \eqref{partial henon equation reduced} corresponding to \eqref{partial henon equation introduction} needs a complete new analysis, in particular for what concerns the study of a {\it{limit problem}} in $\R^{m+1}_+$.

A similar analysis is also used in the proof of Theorem \ref{theorem two point concentration} to study problem \eqref{hyper-plane of black holes} directly, without using any reduction argument.

The outline of this paper is the following. In Section \ref{section reduction} we explain the reduction method. In Section \ref{section radial} we prove Theorem \ref{theorem henon concentration} while in Section  \ref{section partial}, after several preliminary crucial estimates, we prove Theorem \ref{theorem partial henon concentration}. Finally in Section \ref{section two point concentration} we prove Theorem \ref{theorem two point concentration}.

\section{A preliminary reduction lemma}\label{section reduction}
Given an integer $m \geq 1$ we set $\mathcal{G}_m : = \mathcal{O}(m) \times \mathcal{O}(m) \con \mathcal{O}(2m)$. Then $g \in \mathcal{G}_m$ if, and only if, there exist $g_1, g_2 \in \mathcal{O}(m)$ such that
\[
g(y_1, y_2) = (g_1 y_1, g_2 y_2), \quad \forall \, y_1,y_2 \in \R^m.
\]

\begin{definition}
We say that a set  $\Omega \con \R^{2m}$ is invariant by the action of $\mathcal{G}_m$ if 
$g \Omega = \Omega$ for all $g \in \mathcal{G}_m$. Given a function $u: \Omega \rightarrow \R$ defined on an invariant set $\Omega$, we say that $u$ is doubly symmetric if 
\[
u (g(y_1,y_2))  = u(g_1y_1, g_2 y_2) = u(y_1, y_2), \ \forall \, (y_1,y_2) \in \Omega, \ \forall \, g=(g_1, g_2) \in \mathcal{G}_m.
\]
As above, a function $u: \Omega \rightarrow \R$, defined on a invariant set $\Omega \con \R^{2m}$, is said to be doubly symmetric if
\[
u(y_1, y_2) = u (|y_1|, |y_2|), \quad \forall \, (y_1, y_2) \in \Omega.
\] 
\end{definition}

Now we perform a suitable change of variables as in \cite{pacella-srikanth}. Given any point $(y_1, y_2) \in B_{2m}(0,1)$ we write:
\begin{equation}\label{identification 1}
\left\{
\begin{array}{l}
|y_1| = r \cos{\theta};  |y_2| = r \sin{\theta};  r= \sqrt{|y_1|^2 + |y_2|^2}; r \in [0,1), \theta \in \left[ 0, \dfrac{\pi}{2}\right]; \\
r = \sqrt{2 \rho}; \ \theta = \dfrac{\sigma}{2}; \ \rho \in \left[0, \dfrac{1}{2} \right) \; \text{and} \; \sigma \in [0, \pi].
\end{array}
\right.
\end{equation}

Given any doubly symmetric $C^2$-function $u : B_{2m}(0,1) \rt \R$, we write
\begin{equation}\label{identification 2}
u(y_1, y_2) = u(|y_1|, |y_2|) = u(r \cos \theta, r \sin \theta) = u (r, \theta) 
\end{equation}
and we get
\begin{multline*}
\Delta_{2m}u (y_1, y_2)  = u_{rr}(r , \theta) + \dfrac{2m-1}{r}u_r(r,\theta) \\
+ \frac{m-1}{r^2}\left( \dfrac{\cos \theta}{\sin \theta} - \dfrac{\sin \theta}{\cos \theta}\right) u_ {\theta}(r, \theta) + \dfrac{u_{\theta \theta}(r, \theta)}{r^2}.
\end{multline*}
Then we write $v(\rho, \sigma) : = u\left(\sqrt{2 \rho}, \dfrac{\sigma}{2}\right)$, with $r = \sqrt{2 \rho}$ and $\theta = \dfrac{\sigma}{2}$ and we get
\begin{multline*}
\Delta_{2m}u(y_1, y_2) = 2 \rho \left( v_{\rho \rho} (\rho, \sigma) + \dfrac{m}{\rho} v_ {\rho}(\rho, \sigma) \right. \\
\left. + \dfrac{m-1}{\rho^2} \dfrac{\cos \sigma}{\sin \sigma} v_{\sigma}(\rho, \sigma) + \dfrac{v_{\sigma \sigma}(\rho, \sigma)}{\rho^2} \right).
\end{multline*}

Now we recall that, if $v: B_{m+1}(0,1/2) \rt \R$ is an axially symmetric with respect to the axis $\R e_{m+1} \con \R^{m+1}$ $C^2$-function and if we set
\begin{equation}\label{identification 3}
z = (z_1, \ldots, z_{m+1}); \rho = |z|; z_{m+1} = \rho \cos \sigma; \rho \in \left[0, \dfrac{1}{2}\right) \; \text{and} \; \sigma \in [0, \pi],
\end{equation}
then 
\begin{equation}\label{identification 4}
 v(z_1, \ldots, z_m, z_{m+1}) = v(\rho, \sigma)
\end{equation}
and
\begin{multline*}
\Delta_{m+1} v(z_1, \ldots, z_{m+1}) =  v_{\rho \rho} (\rho, \sigma) + \dfrac{m}{\rho} v_ {\rho}(\rho, \sigma) \\
+ \dfrac{m-1}{\rho^2} \dfrac{\cos \sigma}{\sin \sigma} v_{\sigma}(\rho, \sigma) + \dfrac{v_{\sigma \sigma}(\rho, \sigma)}{\rho^2}.
\end{multline*}

At this point we have proved the following lemma; see also \cite[Section 3]{pacella-srikanth}.

\begin{lemma}\label{reduced laplacian}
 There exists a one to one correspondence between doubly \linebreak symmetric $C^2$-function $u : B_{2m} (0,1)\menos\{0\} \rt \R$ and the axially symmetric, with respect to the axis $\R e_{m+1} \con \R^{m+1}$, $C^2$-functions,  $v : B_{m+1}(0,1/2)\menos\{0\} \rt \R$. This correspondence is given by
\begin{multline}\label{identification 5}
 u(y_1, y_2) = u(|y_1|, |y_2|) = u(r \cos \theta, r \sin \theta) = u (r, \theta) \\
  =  u\left(\sqrt{2 \rho}, \dfrac{\sigma}{2}\right) = v(\rho, \sigma) = v(z_1, \ldots, z_{m+1}),
\end{multline}
with the change of variables \eqref{identification 1} and \eqref{identification 3}. Moreover,
\begin{equation} \label{identification 6}
 \Delta_{2m}u (y_1, y_2) = 2 |z| \Delta_{m+1}v(z_1, \ldots, z_m, z_{m+1}). 
\end{equation}
\end{lemma}

We also stress that the changes of variables \eqref{identification 1} and \eqref{identification 3} lead to
\begin{equation}\label{identification 7}
\left.
\begin{array}{r}
 |y_1| = r \cos \theta =  \sqrt{2 \rho} \cos \left( \dfrac{\sigma}{2} \right) = \sqrt{2 |z|} \sqrt{\dfrac{1 + \cos \sigma}{2}}\\
 = \sqrt{|z| + |z| \cos {\sigma}} = \sqrt{| z| +z_{m+1}} \,, \\\\
 \quad |y_2| =  r \sin \theta =  \sqrt{2 \rho} \sin \left( \dfrac{\sigma}{2} \right) = \sqrt{2 |z|} \sqrt{\dfrac{1 - \cos \sigma}{2}}\\
 = \sqrt{|z| - |z| \cos {\sigma}} = \sqrt{| z| - z_{m+1}} \, ,\\\\
 \quad |(y_1, y_2)| = \sqrt{|y_1|^2 + |y_2|^2} = \sqrt{2 |z|} \,.
 \end{array}
 \right\}
\end{equation}

\begin{remark}\label{remark singularity}
 Due to the singularity of $|z|$ at $z=0$, Lemma {\rm\ref{reduced laplacian}} does not hold between doubly symmetric $C^2$-function $u : B_{2m} (0,1)\rt \R$ and the \linebreak axially symmetric with respect to the axis $\R e_{m+1} \con \R^{m+1}$ $C^2$-functions,  $v : B_{m+1}(0,1/2) \rt \R$. Indeed, for each $m\geq 2$ consider
 \[
 u(y_1, y_2) = |(y_1, y_2)|^2, \qquad (y_1, y_2) \in B_{2m}(0,1).
 \]
 Then $u \in C^{\infty}(B_{2m}(0,1))$ and $\Delta_{2m}u(y_1, y_2) = 4m$. The function \linebreak $v:  B_{m+1}(0,1/2) \rt \R$ associated to $u$ is given by
 \[
 v(z_1, \ldots, z_{m}, z_{m+1})= 2 |z|, \qquad z \in  B_{m+1}(0,1/2),
 \]
 which is singular at $z=0$.
\end{remark}

\section{Radially invariant problems and proof of Theorem \ref{theorem henon concentration}}\label{section radial}
In the search of doubly symmetric solutions of 
\begin{equation}\label{radial 2m}
\left\{
\begin{array}{l}
- \Delta_{2m} u = f(|(y_1, y_2)|, u), \quad (y_1, y_2) \in B_{2m}(0,1), \\
u = 0 \quad \text{on} \quad \partial B_{2m}(0,1),
\end{array}
\right.
\end{equation}
we perform the change of variables from Section \ref{section reduction} and we are led  to investigate the existence of axially symmetric, with respect to $\R e_{m+1} \con \R^{m+1}$, solutions of
\begin{equation}\label{radial m+1}
\left\{
\begin{array}{l}
 - \Delta_{m+1} v = \dfrac{f(\sqrt{2|z|}, v)}{2 |z|}, \quad z \in B_{m+1}(0,1/2),\\
 v = 0 \quad \text{on} \quad \partial B_{m+1}(0,1/2).
\end{array}
\right.
\end{equation}

Due to the singularity of $|z|$ at $z=0$ as well as the possible singularity of $f(\sqrt{2|z|}, v)/2 |z|$ at $z= 0$, the claim about the equivalence between problems \eqref{radial 2m} and \eqref{radial m+1} must be careful checked; see Remark \ref{remark singularity} and also  \cite[Theorem 2.3]{ederson-djairo-olimpio-JFA} for a related regularity problem. Note that in \cite{pacella-srikanth} the domains considered are annuli, so the singularity at the origin does not appear. Nevertheless, for the Hénon equation, that is in the case $f(|(y_1, y_2)|, u) = |(y_1, y_2)|^{\alpha}|u|^{p-2}u$, we get it.  Our arguments are based on some regularity results, namely equivalence between weak and classical solutions. In this direction we mention that the classical regularity results as in \cite{agmon-douglis-nirenberg,brezis-kato} does not apply to our problems posed in $\R^{2m}$ since we are working with problems that may be supercritical in the sense that $2 < p < 2(m+1)/(m-1)$ allows $p > 4m/(2m -2)$. 

In order to proceed with \eqref{radial 2m} and \eqref{radial m+1} with $f(|(y_1, y_2)|,u) = |(y_1, y_2)|^{\alpha} |u|^{p-2}u$, that is,
\begin{equation}\label{radial 2m henon}
\left\{
\begin{array}{l}
- \Delta_{2m} u = |(y_1, y_2)|^{\alpha} |u|^{p-2}u, \quad (y_1, y_2) \in B_{2m}(0,1), \\
u = 0 \quad \text{on} \quad \partial B_{2m}(0,1),
\end{array}
\right.
\end{equation}
and 
\begin{equation}\label{radial m+1 henon}
\left\{
\begin{array}{l}
 - \Delta_{m+1} v = |2z|^{\frac{\alpha-2}{2}}|v|^{p-2}v, \quad z \in B_{m+1}(0,1/2),\\
 v = 0 \quad \text{on} \quad \partial B_{m+1}(0,1/2).
\end{array}
\right.
\end{equation}
we need to introduce some notation. Also observe that if we write
\begin{equation}\label{half to complete}
w(z) = \left(\frac{1}{4}\right)^{1/(p-2)}v\left(  \dfrac{z}{2} \right),
\end{equation}
then $v$ is a solution of \eqref{radial m+1 henon} if, and only if, $w$ is a solution of 
\begin{equation}\label{radial m+1 henon unit}
\left\{
\begin{array}{l}
 - \Delta_{m+1} w = |z|^{\frac{\alpha-2}{2}}|w|^{p-2}w, \quad z \in B_{m+1}(0,1),\\
 w = 0 \quad \text{on} \quad \partial B_{m+1}(0,1).
\end{array}
\right.
\end{equation}

\begin{definition}
 Assume $m \geq2$,  $2 < p < \frac{2(m+1)}{m-1}$ and set
 \begin{multline*}
\mathcal{H}_m : = \{u \in H^1_0(B_{2m}(0,1));  \ u(g_1y_1, g_2 y_2) = u(y_1, y_2), \\ \ \forall \, (y_1,y_2) \in B_{2m}(0,1), \ \forall \, g=(g_1, g_2) \in \mathcal{G}_m \},
 \end{multline*}
 with $\mathcal{G}_m$ as defined in Section {\rm \ref{section reduction}}. Then, cf. \cite[Theorem 2.1 and Corollary 2.3]{badiale-serra}, there exists $\alpha_0 = \alpha_{0}(p,m) > 4$ such that $\mathcal{H}_m$ is compactly imbedded in $L^{p}(B_{2m}(0,1), |(y_1, y_2)|^{\alpha})$ for every $\alpha > \alpha_0$. Assume $\alpha > \alpha_0$.
\begin{enumerate}[{\rm1.}]

\item We say that $U$ is a weak doubly symmetric solutions of \eqref{radial 2m henon} if $U$ is a critical point  of the $C^1(\mathcal{H}_m, \R)$-functional
\[
I_m (u) = \dfrac{1}{2}\int_{B_{2m}(0,1)} \!\!\!\!\!\!\!\!\!\!   |\nabla u|^2 d(y_1,y_2) - \dfrac{1}{p}\int_{B_{2m}(0,1)}\!\!\!\!\! \!\!\!\!\!  |(y_1, y_2)|^{\alpha} | u|^{p} d(y_1,y_2), \quad u \in \mathcal{H}_m.
\]

 \item We say that $u_{\alpha} \in \mathcal{H}_m$ is a least energy solution among the doubly \linebreak symmetric solutions of \eqref{radial 2m henon} if $u_{\alpha}$ is a nontrivial doubly symmetric solution of \eqref{radial 2m henon} and
\[
I_{m}(u_{\alpha}) = \min \{ I_{m}(u); \; u \; \text{is a nontrivial doubly symmetric sol. of \eqref{radial 2m henon}}  \}.
\]

\item We say that $W$ is a weak solution of \eqref{radial m+1 henon unit} if $W$ is a critical point  of the $C^1(H^1_0(B_{m+1}(0,1), \R)$-functional
\[
J(w) = \dfrac{1}{2}\int_{B_{m+1}(0,1)}\!\!\!\!\!\!\!\!\!\! |\nabla w|^2 dz - \dfrac{1}{p}\int_{B_{m+1}(0,1)} \!\!\!\!\!\!\!\!\!\!|z|^{\frac{\alpha-2}{2}} | w|^{p} dz, \quad w \in H^1_0(B_{m+1}(0,1)).
\]

\item We say that $w_{\alpha} \in H^1_0(B_{m+1}(0,1))$ is a least energy solution of \eqref{radial m+1 henon unit} if $w_{\alpha}$ is a nontrivial solution of \eqref{radial m+1 henon unit} and
\[
J(w_{\alpha}) = \min \{ J(w); \; w \; \text{is nontrivial solution of \eqref{radial m+1 henon unit}}  \}.
\]

\end{enumerate}
\end{definition}

\begin{lemma}[{\bf{\cite[Propositions 5.4 and 5.5]{ederson-djairo-olimpio-JFA}}}]\label{lemma deo}
 Assume $m \geq2$ and  $2 < p < \frac{2(m+1)}{m-1}$. There exists $\alpha_0 = \alpha_{0}(p,m) > 4$ such that for every $\alpha >\alpha_0$, $u$ is a weak doubly symmetric solution of \eqref{radial 2m henon} if, and only if, $u$ is a classical doubly symmetric solutions of \eqref{radial 2m henon}. In this case, $u \in C^{2, \gamma}(\overline{B}_{2m}(0,1))$ for all $0 < \gamma < 1$.
\end{lemma}

\begin{proposition}\label{equivalence henon}
Assume $m \geq2$,  $2 < p < \frac{2(m+1)}{m-1}$. Then there exists $\alpha_0 = \alpha_{0}(p,m) > 4$ such that for every $\alpha > \alpha_0$, \eqref{identification 5} provides a bijective correspondence between  
\[
X = \{ u \in C^2(\overline{B}_{2m}(0,1)); \; u \; \text{is a doubly symmetric classical solution of} \; \eqref{radial 2m henon} \}
\]
and
\begin{multline*}
Y = \left\{ v \in C^2(\overline{B}_{m+1}(0,1/2)); \; v \; \text{is an axially symmetric}, \right.\\
\left. \text{w.r.t. $\R e_{m+1} \con \R^{m+1}$, classical solution of} \; \eqref{radial m+1 henon} \right\}.
\end{multline*}
In addition, any $u \in X$ and any $v \in Y$ are such that $u \in C^{2, \gamma}(\overline{B}_{2m}(0,1))$, $v \in C^{2, \gamma}(\overline{B}_{m+1}(0,1/2))$ for all $0 < \gamma < 1$.
\end{proposition}
\begin{proof}
Let $u \in X$. Then $u$ is a weak doubly symmetric solution of \eqref{radial 2m henon}. So, after changing variables, we get that the function $v$, associated to $u$ by \eqref{identification 5}, is a weak solution of \eqref{radial m+1 henon} in the sense of $H^1_0(B_{m+1}(0,1/2))$; we have also used the classical result of Palais \cite{palais}. Hence, since we have subcritical growth for the problem posed in $B_{m+1}(0,1/2) \con \R^{m+1}$, we apply \cite{agmon-douglis-nirenberg} to get that  $v$ is a classical solution of \eqref{radial m+1 henon}.

On the other hand, let $v \in Y$. Then, after changing variables, we get that the function $u$, associated to $v$ by \eqref{identification 5}, is a weak doubly symmetric solution of \eqref{radial 2m henon}, hence classical by Lemma \ref{lemma deo}. 
\end{proof}

We mention that is proved in \cite{badiale-serra} that the Hénon equation has doubly symmetric solutions, that are non radially symmetric, in case $2 < p < 2(m+1)/(m-1)$ and $\alpha$ is sufficiently large.

Now, from \cite{pacella,smets-su-willem,smets-willem,BW,cao-peng-yan}, we collect  some results about the least energy solutions of \eqref{radial m+1 henon unit}.

\begin{proposition}\label{results about henon}
 Assume $m \geq2$,  $2 < p < \frac{2(m+1)}{m-1}$. Then there exists $\alpha_0 = \alpha_{0}(p,m) > 4$ such that for each $\alpha > \alpha_0$, any least energy solution $w_{\alpha}$ of \eqref{radial m+1 henon unit} (up to rotation and up to replacing $w_{\alpha}$ by $-w_{\alpha}$) is such that:
\begin{enumerate}[(i)]
 \item $w_{\alpha} > 0$ in $B_{m +1}(0,1)$; $w_{\alpha}$ is not radially symmetric; $w_{\alpha}$ is Schwarz foliated w.r.t. the vector $e_{m+1} \in \R^{m+1}$, in particular $w_{\alpha}$ is axially \linebreak symmetric w.r.t. $\R e_{m+1} \con \R^{m+1}$.
 \item $w_{\alpha}$ concentrates at the point $(0,\ldots, 0,1)$ as $\alpha \rt \infty$. In addition, let $0 < \tau_{\alpha} < 1$ be such that
 \[
 \mathcal{M}'_{\alpha}  = \max_{z \in B_{m+1}(0,1)} w_{\alpha}(z) = w_{\alpha}((0, \ldots, 0,\tau_{\alpha})).
 \]
 Then $\alpha (1- \tau_{\alpha}) \rt \ell$ for some positive number $l$ and $\mathcal{M}'_{\alpha} \approx \alpha^{2/(p-2)}$ as $\alpha \rt \infty$.
\end{enumerate}
\end{proposition}

\begin{proposition}\label{relation least energy solutions}
Assume $m \geq2$,  $2 < p < \frac{2(m+1)}{m-1}$. Then there exists $\alpha_0 = \alpha_{0}(p,m) > 4$ such that for each $\alpha > \alpha_0$, $u_{\alpha}$ is a least energy solution among the doubly symmetric solutions of \eqref{radial 2m henon} if, and only if, $w_{\alpha}$ is a least energy solution of \eqref{radial m+1 henon unit} and $u_{\alpha}$ and $w_{\alpha}$ are related by \eqref{identification 5} and \eqref{half to complete}.
\end{proposition}
\begin{proof}
It follows from the change of variables involving $u_{\alpha}$ and $w_{\alpha}$ by means of \eqref{identification 5} and \eqref{half to complete}. 
\end{proof}

\begin{corollary} \label{doubly symmetric solutions henon}
 Assume $m \geq2$,  $2 < p < \frac{2(m+1)}{m-1}$.  Then there exists $\alpha_0 = \alpha_{0}(p,m) > 4$ such that for each $\alpha > \alpha_0$ the following holds. Let $u_{\alpha}$ be a least energy solution among the doubly symmetric solutions of \eqref{radial 2m henon}. Then, up to replacing $u_{\alpha}$ by $-u_{\alpha}$, one has $u_{\alpha} > 0$ in $B_{2m}(0,1)$, $u_{\alpha}$ is not radially symmetric, there exists $0 < r_{\alpha} < 1$ and $\theta_{*} \in \left\{ 0 , \dfrac{\pi}{2} \right\}$ such that
 \[
 \max_{(y_1, y_2) \in B_{2m}(0,1)} u_{\alpha}(y_1, y_2) = u_{\alpha}(r_{\alpha}, \theta_{*}).
 \]
\end{corollary}

We stress that the above corollary guarantees that, up to replacing $u_{\alpha}(y_1, y_2)$ by $\overline{u}_{\alpha}(y_1, y_2) : = u_{\alpha}(y_2, y_1)$, we have
 \[
 \max_{(y_1, y_2) \in B_{2m}(0,1)} u_{\alpha}(y_1, y_2) = u_{\alpha}(y_1,0), \quad |y_1| = r_{\alpha}.
 \]
 
\begin{proof}[Proof of Theorem \ref{theorem henon concentration}]
 It is a straightforward consequence of Propositions \ref{results about henon}, \ref{relation least energy solutions} and Corollary \ref{doubly symmetric solutions henon}.  
  \end{proof}

\section{Partially symmetric problems and proof of Theorem \ref{theorem partial henon concentration}} \label{section partial}
In the search of doubly symmetric solutions of 
\begin{equation}\label{partial 2m}
\left\{
\begin{array}{l}
- \Delta_{2m} u = f(|y_1|, |y_2|, u), \quad (y_1, y_2) \in B_{2m}(0,1), \\
u = 0 \quad \text{on} \quad \partial B_{2m}(0,1),
\end{array}
\right.
\end{equation}
we perform the change of variables from Section \ref{section reduction}, see \eqref{identification 1}, \eqref{identification 3}, \eqref{identification 7}, and we are led  to investigate the existence of axially symmetric, with respect to $\R e_{m+1} \con \R^{m+1}$, solutions of
\begin{equation}\label{partial y1 y2 m+1}
\left\{
\begin{array}{l}
- \Delta_{m+1} v = \dfrac{f(\sqrt{| z| +z_{m+1}}, \sqrt{| z| - z_{m+1}}, v )}{2 |z|}, \ z \in B_{m+1}(0,1/2), \\ \\
 v = 0 \quad \text{on} \quad \partial B_{m+1}(0,1/2).
 \end{array}
\right.
\end{equation}

In this part we consider the particular problem
\begin{equation}\label{partial henon equation}
 \left\{
\begin{array}{l}
- \Delta_{2m} u = |y_2|^{\alpha} |u|^{p-2}u, \quad (y_1, y_2) \in B_{2m}(0,1), \\
u = 0 \quad \text{on} \quad \partial B_{2m}(0,1).
\end{array}
\right.
\end{equation}

Applying the moving planes technique \cite{gidas-ni-nirenberg} we know that any positive classical solution of \eqref{partial henon equation} is such that $u(y_1, y_2) = u (|y_1|, y_2)$ and, for each $y_2$, $u(y_1, y_2)$ is decreasing with respect to $|y_1|$. Therefore, if we look for positive doubly symmetric solutions of \eqref{partial henon equation} we obtain that for any such solution, there exists $0 \leq r < 1$ such that
\[
\gd{\max_{(y_1, y_2) \in B_{2m}(0,1)} u(y_1, y_2) = u(0, y_2), \quad \forall \; y_2 \in S^{m-1}_r},
\]
with $S^{m-1}_r$ as defined in \eqref{sphere radius r in rk}.

From now on in this section we will proceed to prove Theorem \ref{theorem partial henon concentration}.

First, by arguing similarly to the proof of Proposition \ref{equivalence henon}, we can prove the following equivalence. 

\begin{proposition}\label{equivalence partial henon}
Assume $m \geq2$,  $2 < p < \frac{2(m+1)}{m-1}$. Consider \eqref{partial henon equation} and 
\begin{equation}\label{partial henon equation reduced}
\left\{
\begin{array}{l}
 - \Delta_{m+1} v = \dfrac{(|z| - z_{m+1})^{\frac{\alpha}{2}}}{2 |z|}|v|^{p-2}v, \quad z \in B_{m+1}(0,1/2),\\
 v = 0 \quad \text{on} \quad \partial B_{m+1}(0,1/2).
\end{array}
\right.
\end{equation}
Then there exists $\alpha_0 = \alpha_{0}(p,m) > 4$ such that for every $\alpha > \alpha_0$ \eqref{identification 5} provides a bijective correspondence between  
\[
X = \{ u \in C^2(\overline{B}_{2m}(0,1)); \; u \; \text{is a doubly symmetric classical solution of} \; \eqref{partial henon equation} \}
\]
and
\begin{multline*}
Y = \left\{ v \in C^2(\overline{B}_{m+1}(0,1/2)); \; v \; \text{is an axially symmetric},\right.\\
\left.\text{w.r.t. $\R e_{m+1} \con \R^{m+1}$, classical solution of} \; \eqref{partial henon equation reduced}\right\}.
\end{multline*}
In addition, any $u \in X$ and any $v \in Y$ are such that $u \in C^{2, \gamma}(\overline{B}_{2m}(0,1))$, $v \in C^{2, \gamma}(\overline{B}_{m+1}(0,1/2))$ for all $0 < \gamma < 1$.
\end{proposition}

We recall that in the proof of Proposition \ref{equivalence henon} we used \cite[Propositions 5.4 and 5.5]{ederson-djairo-olimpio-JFA}, which assert about classical regularity of weak doubly symmetric solutions of the Hénon equation. The proof of Proposition \ref{equivalence partial henon} follows as the proof of Proposition \ref{equivalence henon}, if we replace Lemma \ref{lemma deo} by \cite[Theorem 2.5]{ederson-pedro-leonelo}, which in particular guarantees classical regularity of weak doubly symmetric solutions of \eqref{partial henon equation}.

We mention that, as in Proposition \ref{relation least energy solutions}, we can show the correspondence between least energy solutions among the doubly symmetric solutions of \eqref{partial henon equation} and least energy solutions among the axially symmetric, with respect to $\R e_{m+1} \con \R^{m+1}$, solutions of \eqref{partial henon equation reduced}. We then turn our attention to \eqref{partial henon equation reduced}. Observe that for every $\alpha >2$
\[
 \dfrac{(|z| - z_{m+1})^{\alpha/2}}{2 |z|} \leq (|z| - z_{m+1})^{(\alpha- 2)/2} \leq 1 \quad \forall \, z \in B_{m+1}(0,1/2) \menos\{0\} 
 \]
and
\[
\lim_{z \rt 0}  \dfrac{(|z| - z_{m+1})^{\alpha/2}}{2 |z|} = 0.
\]

Let $v$ be a positive and axially symmetric, with respect to $\R e_{m+1} \con \R^{m+1}$, solution of \eqref{partial henon equation reduced}. By Proposition \ref{equivalence partial henon}, if $u$ is associated to $v$ by means of \eqref{identification 5}, then $u$ is a positive doubly symmetric solution of \eqref{partial henon equation}. Then as observed before, by the moving planes technique, there exists $0 \leq r < 1$ such that
  \[
 \max_{(y_1, y_2) \in B_{2m}(0,1)} u(y_1, y_2) = u(0, y_2), \quad \forall \, y_2 \in S^{m-1}_{r}.
 \]
Then, with $\rho = \dfrac{r^2}{2}$, we have that
\[
\max_{z \in B_{m+1}(0,1/2)} v (z) = v(0, \ldots, 0, - \rho).
\]

Now let $v_{\alpha}$ be a least energy solution among the axially symmetric ones with respect to $\R e_{m+1} \con \R^{m+1}$, solutions of \eqref{partial henon equation reduced}. Then, up to a multiplicative constant, by the principle of symmetric criticality \cite{palais}, we characterize such solution as a minimizer of a Rellich quotient among the functions in $H^1_0(B_{m+1}(0,1/2))$ invariant by the action of the group
\begin{multline*}
\boldsymbol{G}_m = \left\{ \sigma \in \mathcal{O}(m+1); \exists \; g \in \mathcal{O}(m) \right.\\
\left. \text{s.t.} \; \sigma(z_1, \ldots z_m, z_{m+1}) = (g(z_1, \ldots, z_m), z_{m+1}) \right\}.
\end{multline*}
We can assume that $v_{\alpha} > 0$ in $B_{m+1}(0,1/2)$. So arguing as in the previous paragraph, there exists $0 \leq \rho_{\alpha} < \dfrac{1}{2}$ such that
\begin{equation}\label{maximum value partial henon equation}
\mathfrak {M}_{\alpha} : =\max_{z \in B_{m+1}(0,1/2)} v_{\alpha} (z) = v_{\alpha}(0, \ldots, 0, - \rho_{\alpha}).
\end{equation}

Let 
\begin{equation}\label{w and v}
w_{\alpha} (z) = \left(\dfrac{1}{4}\right)^{1/(p-2)}v_{\alpha}\left(\dfrac{z}{2}\right).
\end{equation} 
Then $w_{\alpha} > 0$ in $B_{m+1}(0,1)$ and $w_{\alpha}$ is a least energy solution among the axially symmetric, with respect to $\R e_{m+1} \con \R^{m+1}$, solutions of
\begin{equation}\label{partial henon equation reduced unit ball}
\left\{
\begin{array}{l}
 - \Delta_{m+1} w = h_{\alpha}(z)|w|^{p-2}w, \quad z \in B_{m+1}(0,1),\\
 w = 0 \quad \text{on} \quad \partial B_{m+1}(0,1),
\end{array}
\right.
\end{equation}
with
\[
h_{\alpha}(z) := \dfrac{\left(\dfrac{|z| - z_{m+1}}{2}\right)^{\frac{\alpha}{2}}}{ |z|}, \quad z \in B_{m+1}(0,1).
\]

Also observe that for every $\alpha > 2$
\begin{equation}\label{ineq weights}
h_{\alpha}(z) = \dfrac{\left(\dfrac{|z| - z_{m+1}}{2}\right)^{\frac{\alpha}{2}}}{ |z|} < |z|^{\frac{\alpha -2}{2}} \quad \forall \, z \in B_{m+1}(0,1)\menos\{0\}. 
\end{equation}

Now we compare \eqref{partial henon equation reduced unit ball} and
\begin{equation}\label{aux henon equation}
\left\{
\begin{array}{l}
 - \Delta_{m+1} \psi =  |z|^{\frac{\alpha -2}{2}}|\psi|^{p-2}\psi, \quad z \in B_{m+1}(0,1),\\
 \psi = 0 \quad \text{on} \quad \partial B_{m+1}(0,1).
\end{array}
\right.
\end{equation}

We set
\[
H_m : = \{ w  \in H^1_ 0(B_{m+1}(0,1)); g u = u \; \forall \; g \in \boldsymbol{G}_m \},
\]
the space of functions in $H^1_ 0(B_{m+1}(0,1))$ that are axially symmetric with respect to $\R e_{m+1} \con \R^{m+1}$. We also set
\[
\gd{ S_{\alpha,p} : = \inf_{\psi \in H^1_0(B_{m+1}(0,1)) \menos \{ 0\}} \dfrac{\int | \nabla \psi|^2 dz}{\left(\int |z|^{\frac{\alpha -2}{2}} |\psi|^p dz \right)^{2/p}} } 
\]
and
\[
\gd{S'_{\alpha, p} : =  \inf_{w \in H_m\menos \{ 0\}} \dfrac{\int | \nabla w|^2 dz}{ \left( \int h_{\alpha}(z) |w|^p dz\right)^{2/p}} }.
\]

Then, from \cite{smets-willem,pacella}, we have that any minimizer $\psi$ of $S_{\alpha,p}$, up to rotation, is such that $\psi \in H_m$. Then, from \eqref{ineq weights} we conclude that 
\begin{equation}\label{sobolev constants 1}
S'_{\alpha, p} > S_{\alpha,p} \quad \text{for every}  \quad \alpha > 2.
\end{equation}

We recall that
\begin{equation}\label{asymptotic henon}
 \dfrac{S_{\alpha,p}}{\alpha^{[2(m+1) - p(m-1)]/p}} = \dfrac{m_{1,p}}{2^{[2(m+1) - p(m-1)]/p}} + o(1) \quad \text{as} \quad \alpha \rt \infty,
\end{equation}
where
\[
m_{\gamma, p} = \inf \left\{ \int | \nabla w |^2 dz; w \in \mathcal{D}^{1,2}_0(\R^{m+1}_+), \ \int_{\R^{m+1}_+} e^{-\gamma z_{m+1}} |w|^p dz =1\right\},
\]
which is attained for every $\gamma>0$ and $2 < p < \frac{2(m+1)}{m-1}$; see \cite[Theorem 2.1 and Remark 4.8]{cao-peng-yan}. In particular, from \eqref{asymptotic henon}, there exist $C_1, C_2 >0$ such that
\begin{equation}\label{henon from below and above}
C_1 \alpha^{[2(m+1) - p(m-1)]/p} \leq S_{\alpha,p} \leq C_2 \alpha^{[2(m+1) - p(m-1)]/p} \quad \text{as} \quad \alpha \rt \infty.
\end{equation}

Moreover, the equation
\begin{equation}\label{limit problem henon}
-\Delta w = e^{-z_{m+1}} |w|^{p-2} w \quad \text{in} \quad \R^{m+1}_+
\end{equation}
is called the limit problem associated to \eqref{aux henon equation}, since after suitable rescaling, as showed in \cite{cao-peng-yan},  least energy solutions of \eqref{aux henon equation} converge to least energy solutions of \eqref{limit problem henon} as $\alpha \rt \infty$.

Next we prove that $S'_{\alpha,p}$ may also be controlled as in \eqref{henon from below and above}. Indeed we show that the limit problem associated to \eqref{partial henon equation reduced unit ball} is a slight variation of \eqref{limit problem henon}.

\begin{proposition}\label{proposition asymptotic partial henon reduced}
There holds
\begin{equation}\label{asymptotic partial henon reduced}
 \dfrac{S'_{\alpha,p}}{\alpha^{[2(m+1) - p(m-1)]/p}} = m_{1/2,p} + o(1) \quad \text{as} \quad \alpha \rt \infty.
\end{equation}
\end{proposition}

We prove some preliminary lemmas in order to go through the proof of Proposition \ref{proposition asymptotic partial henon reduced}.

\begin{lemma}
There exist $C_1, C_2$ positive constants such that
\begin{equation}\label{partial from below and above}
C_1 \alpha^{[2(m+1) - p(m-1)]/p} \leq S'_{\alpha,p} \leq C_2 \alpha^{[2(m+1) - p(m-1)]/p} \quad \text{as} \quad \alpha \rt \infty.
\end{equation}

\end{lemma} 
\begin{proof}
Given $\epsilon >0$, choose $w_{\epsilon} \in C_c^{\infty}(\R^{m+1}_+)$ such that, $w_{\epsilon} \neq 0$, $w_{\epsilon}$ is axially symmetric with respect to $\R e_{m+1} \con \R^{m+1}$ and
\[
\dfrac{\int_{\R^{m+1}_+} |\nabla w_{\epsilon}(s)|^2 ds}{ \left( \int_{\R^{m+1}_+} e^{- (s_{m+1}/2)} |w_{\epsilon}(s)|^p ds \right)^{2/p}} < m_{1/2,p} + \epsilon.
\]

Set
\[
\overline{w}_{\epsilon} (z) = w_{\epsilon} (\alpha z', \alpha[ (1- |z'|^2)^{1/2} + z_{m+1}]), \quad z = (z', z_{m+1}) \in B_{m+1}(0,1).
\]
Then, it is easy to see that $\overline{w}_{\epsilon} \in H_m$ for any large $\alpha$.

We will perform the change of variables for $x = (x', x_{m+1})\in \R^m \times \R$, $s = (s', s_{m+1}) \in \R^m \times \R$:
\begin{equation}\label{change of variables reduced}
x = \alpha e_{m+1} + \alpha z \ \text{and} \ s' = x', \, s_{m+1} = x_{m+1} + \alpha (-1 + (1 - \alpha^{-2} |x'|^2)^{1/2}). 
\end{equation}
Then, since $w_{\epsilon}$ has compact support in $\R^{m+1}_+$ we get:
\begin{multline*}
\int_{B_{m+1}(0,1)} \!\!\!\!\!\!\!\!\!| \nabla \overline{w}_{\epsilon}|^2 dz = \alpha^2 \int_{B_{m+1}(0,1)} \!\! \left\{\sum_{i =1}^{m} \left[\left|\partial_i w_{\epsilon}(\alpha z', \alpha[ (1- |z'|^2)^{1/2} + z_{m+1}]) \right. \right. \right. \\
 \left. \left. - \frac{z_i}{(1- |z'|^2)^{1/2}}\partial_{m+1}w_{\epsilon}(\alpha z', \alpha[ (1- |z'|^2)^{1/2} + z_{m+1}])\right|^2 \right] \\
\left.+ \left| \partial_{m+1} w_{\epsilon}(\alpha z', \alpha[ (1- |z'|^2)^{1/2} + z_{m+1}]) \right|^2 \right\}dz \\
= \alpha^{1-m}\int_{B_{m+1}(\alpha e_{m+1}, \alpha)}  \left\{\sum_{i =1}^{m} \left[\left|\partial_i w_{\epsilon}(x', (\alpha^2- |x'|^2)^{1/2} + x_{m+1} -\alpha) \right. \right. \right. \\
 \left. \left. - \frac{\alpha^{-1} x_i}{(1 - \alpha^{-2}|x'|^2)^{1/2}}\partial_{m+1}w_{\epsilon}(x', (\alpha^2- |x'|^2)^{1/2} + x_{m+1} -\alpha)\right|^2 \right]\\
 \left. + \left| \partial_{m+1} w_{\epsilon}(x', (\alpha^2- |x'|^2)^{1/2} + x_{m+1} -\alpha) \right|^2 \right\}dx\\
= \alpha^{1-m} \int_{\R^{m+1}_+} \left\{\sum_{i =1}^{m} \left[\left|\partial_i w_{\epsilon}(s) - \frac{\alpha^{-1} s_i}{(1 - \alpha^{-2}|s'|^2)^{1/2}}\partial_{m+1}w_{\epsilon}(s)\right|^2 \right] \right.\\
\left. + \left| \partial_{m+1} w_{\epsilon}(s) \right|^2 \right\}ds = \alpha^{1-m}\left[ \int_{\R^{m+1}_+} | \nabla w_{\epsilon}(s)|^2 ds + O(\alpha^{-1})  \right].
\end{multline*}

On the other hand, by the change of variables \eqref{change of variables reduced}, we have
\begin{multline*}
 h_{\alpha}(z) = \left. \left( \frac{|x - \alpha e_{m+1}| - (x_{m+1} - \alpha)}{2\alpha} \right)^{\alpha/2} \middle/ \left| \frac{x}{\alpha} - e_{m+1} \right| \right. \\
 = \left( \frac{\sqrt{|s'|^2 + (s_{m+1} - (\alpha^2 - |s'|^2)^{1/2})^2}  - (s_{m+1} - (\alpha^2 - |s'|^2)^{1/2})}{2\alpha}\right)^{\alpha/2}  \cdot \\
  \cdot \left(\left|\frac{s'}{\alpha}\right|^2 + \left(\frac{s_{m+1}}{\alpha} - \left(1 - \left|\frac{s'}{\alpha}\right|^2\right)^{1/2}\right)^2\right)^{-1/2}.
\end{multline*}

Now, if $s \in supp \,w_{\epsilon}$, then
\[
\sqrt{ \left|\frac{s'}{\alpha}\right|^2 + \left(\frac{s_{m+1}}{\alpha} - \left(1 - \left|\frac{s'}{\alpha}\right|^2\right)^{1/2}\right)^2} = 1 + O(\alpha^{-1})
\]
and
\begin{multline*}
\frac{\sqrt{|s'|^2 + (s_{m+1} - (\alpha^2 - |s'|^2)^{1/2})^2}  - (s_{m+1} - (\alpha^2 - |s'|^2)^{1/2})}{2\alpha} \\
= - \frac{s_{m+1}}{\alpha} + 1 + \frac{\sqrt{|s'|^2 + (s_{m+1} - (\alpha^2 - |s'|^2)^{1/2})^2}  +(s_{m+1} - (\alpha^2 - |s'|^2)^{1/2})}{2\alpha} \\
+ \frac{(\alpha^2 - |s'|^2)^{1/2} - \alpha}{\alpha} = 1 - \frac{s_{m+1}/2}{\alpha/2} + O(\alpha^{-2}).
\end{multline*}

Then, if $s \in supp \, w_{\epsilon}$ we have,
\begin{equation}\label{exponential}
h_{\alpha}(z) = e^{-\frac{s_{m+1}}{2} + O(\alpha^{-1})} + O(\alpha^{-1})
\end{equation}
and so
\begin{multline*}
\int_{B_{m+1}(0,1)} \!\!\!\!\!\!\!\!\!\!\!\!h_{\alpha}(z) \overline{w}^p_{\epsilon}(z) dz = \alpha^{-(m+1)} \left[ \int_{\R^{m+1}_+} \!\!\!\!e^{-\frac{s_{m+1}}{2}  + O(\alpha^{-1})} w_{\epsilon}^p(s)ds + O(\alpha^{-1})  \right] \\
= \alpha^{-(m+1)} \left[ \int_{\R^{m+1}_+} e^{-\frac{s_{m+1}}{2}} w_{\epsilon}^p(s)ds + O(\alpha^{-1})  \right].
\end{multline*}

Hence, by the definition of $S'_{\alpha,p}$, we have
\begin{multline*}
S'_{\alpha, p} \leq \alpha^{[2(m+1) - p(m-1)]/p}\dfrac{\int_{\R^{m+1}_+} | \nabla w_{\epsilon}|^2 ds + O(\alpha^{-1})}{\left( \int_{\R^{m+1}_+} e^{-\frac{s_{m+1}}{2}} w_{\epsilon}^p(s)ds + O(\alpha^{-1}) \right)^{2/p}} \\
= \alpha^{[2(m+1) - p(m-1)]/p}\dfrac{\int_{\R^{m+1}_+} | \nabla w_{\epsilon}|^2 ds }{\left( \int_{\R^{m+1}_+} e^{-\frac{s_{m+1}}{2}} w_{\epsilon}^p(s)ds \right)^{2/p} + O(\alpha^{-1})} \\
\leq m_{1/2,p} + \epsilon + O(\alpha^{-1}).
\end{multline*}
From \eqref{sobolev constants 1}, \eqref{henon from below and above} and the last inequality we have that there exist $C_1>0$ such that
\begin{equation}\label{almost sharp}
C_1 \leq \dfrac{S'_{\alpha,p}}{\alpha^{[2(m+1) - p(m-1)]/p}} \leq m_{1/2,p} + o(1) \quad \text{as} \quad \alpha \rt \infty.
\end{equation} 
\end{proof}

\medskip

Let $w_{\alpha} > 0$ be a least energy solution among the axially symmetric, with respect to $\R e_{m+1} \con \R^{m+1}$, solutions of \eqref{partial henon equation reduced unit ball}. Then
\[
\int_{B_{m+1}(0,1)} | \nabla w_{\alpha}|^2 dz = \int_{B_{m+1}(0,1)} h_{\alpha}(z) w_{\alpha}^p dz = \left(S'_{\alpha,p}\right)^{p/{p-2}}
\]
and from \eqref{partial from below and above}, there exist $C_1, C_2$ positive constants such that
\begin{multline}\label{gradient w}
C_1 \alpha^{[2(m+1) - p(m-1)]/(p-2)} \leq \int_{B_{m+1}(0,1)} | \nabla w_{\alpha}|^2 dz = \int_{B_{m+1}(0,1)} h_{\alpha}(z) w_{\alpha}^p dz \\
\leq C_2  \alpha^{[2(m+1) - p(m-1)]/(p-2)}
\end{multline}
as $\alpha \rt \infty$.

Set
\[
\overline{w}_{\alpha}(z) = \alpha^{-2/(p-2)} w_{\alpha} \left( \frac{z}{\alpha}\right), \quad z \in B_{m+1}(0, \alpha).
\]
Then we have
\[
-\Delta \overline{w}_{\alpha}  = h_{\alpha} \left( \frac{z}{\alpha}\right) (\overline{w}_{\alpha})^{p-1}, \ z \in B_{m+1}(0, \alpha) \ \text{with} \  \overline{w}_{\alpha} = 0 \ \text{on} \ \partial B_{m+1}(0, \alpha)
\]
and
\[
\int_{B_{m+1}(0, \alpha)} | \nabla \overline{w}_{\alpha}|^2 dz = \alpha^{- [2(m+1) - p(m-1)]/(p-2) }\int_{B_{m+1}(0,1)} | \nabla w_{\alpha}|^2 dz 
\]
and hence
\begin{equation}\label{bound of gradient}
C_1 \leq \int_{B_{m+1}(0, \alpha)} | \nabla \overline{w}_{\alpha}|^2 dz \leq C_2 \quad \text{as} \quad \alpha \rt \infty.
\end{equation}

On the other hand, as proved in \cite[pp. 473 and 474]{cao-peng-yan}, there exists $C_{3}>0$ such that
\[
\inf_{u \in H^1_0(B_{m+1}(0,1)), \ u \neq 0} \dfrac{\int |\nabla u|^2 dz}{\int |z|^{(\alpha-2)/2} u^2 dz} \geq C_3 \alpha^2 \quad \text{as} \quad \alpha \rt \infty.
\]
As a consequence
\[
\dfrac{\int |\nabla w_{\alpha}|^2 dz}{\int h_{\alpha}(z) w_{\alpha}^2 dz} >  \dfrac{\int |\nabla w_{\alpha}|^2 dz}{\int |z|^{(\alpha-2)/2} w_{\alpha}^2 dz} \geq C_ 3 \alpha^2 \quad \text{as} \quad \alpha \rt \infty.
\]
Then we combine \eqref{gradient w} and the last inequality to get
\begin{multline}\label{l2}
\int_{B_{m+1}(0, \alpha)} h_{\alpha} \left( \frac{z}{\alpha}\right) \overline{w}_{\alpha}^2 (z) dz = \alpha^{m+1 - \frac{4}{p-2}} \int_{B_{m+1}(0,1)} h_{\alpha}(z) w_{\alpha}^2(z) dz \\
 \leq C \alpha^{[p(m-1) - 2(m+1)]/(p-2)} \int_{B_{m+1}(0,1)} | \nabla w_{\alpha}|^2 dz \leq C \quad \text{as} \quad \alpha \rt \infty.
\end{multline}

\begin{lemma}\label{partial maximum value asymptotic}
 There exist $C_1, C_2$ positive constants such that
 \begin{equation}\label{E1}
 C_1 \leq \max_{z \in B_{m+1}(0, \alpha)}\overline{w}_{\alpha} (z) \leq C_2 \quad \text{as} \quad \alpha \rt \infty,
 \end{equation}
 that is,  \begin{equation}\label{E2}
 C_1 \alpha^{2/(p-2)}\leq \max_{z \in B_{m+1}(0, 1)}w_{\alpha} (z) \leq C_2 \alpha^{2/(p-2)} \quad \text{as} \quad \alpha \rt \infty.
 \end{equation}
\end{lemma}
\begin{proof}
From \eqref{bound of gradient} and \eqref{l2}, it follows that
\begin{multline*}
0 < C_1 \leq \int_{B_{m+1}(0, \alpha)} | \nabla \overline{w}_{\alpha}|^2 dx =  \int_{B_{m+1}(0, \alpha)} h_{\alpha}\left( \frac{z}{\alpha} \right) (\overline{w}_{\alpha})^p dz \\
\leq \max_{z \in B_{m+1}(0,\alpha)} (\overline{w}_{\alpha})^{p-2}  \int_{B_{m+1}(0, \alpha)} h_{\alpha}\left( \frac{z}{\alpha} \right) (\overline{w}_{\alpha})^2 dz \leq C \max_{z \in B_{m+1}(0,\alpha)} (\overline{w}_{\alpha})^{p-2} 
\end{multline*}
as $\alpha \rt \infty$.

Now we prove the reverse inequality. By contradiction, suppose that $(\| \overline{w}_{\alpha} \|_{\infty})$ is not bounded from above as $\alpha \rt \infty$. Then there exists a sequence $(\alpha_n)$ such that $\| \overline{w}_{\alpha_n} \|_{\infty} \rt \infty$ and $\alpha_n \rt \infty$ as $n \rt \infty$. Let $z_{\alpha_n} \in B_{m+1}(0, \alpha_n)$ such that $\| \overline{w}_{\alpha_n} \|_{\infty} = \overline{w}_{\alpha_n} (z_{\alpha_n})$ and set
\[
\overline{v}_{\alpha_n} (z) = \frac{1}{\| \overline{w}_{\alpha_n} \|_{\infty}} \overline{w}_{\alpha_n} \left( \| \overline{w}_{\alpha_n} \|_{\infty}^{-(p-2)/2} z+ z_{\alpha_n}\right), 
\]
for $z \in B_{m+1}(- z_{\alpha_n} \| \overline{w}_{\alpha_n} \|_{\infty}^{(p-2)/2}, \alpha_n \| \overline{w}_{\alpha_n} \|_{\infty}^{(p-2)/2} )$.
Then
\[
-\Delta \overline{v}_{\alpha_n} = h_{\alpha_n} \left( \frac{\| \overline{w}_{\alpha_n} \|_{\infty}^{-(p-2)/2} z }{\alpha_n} + z_{\alpha_n}\right) (\overline{v}_{\alpha_n})^{p-1} 
\]
in $B_{m+1}(- z_{\alpha_n} \| \overline{w}_{\alpha_n} \|_{\infty}^{(p-2)/2}, \alpha_n \| \overline{w}_{\alpha_n} \|_{\infty}^{(p-2)/2} )$, with homogenous Dirichlet \linebreak boundary condition. Observe that, from \eqref{bound of gradient} and $2 < p < \frac{2(m+1)}{m-1}$, it \linebreak follows that
\[
\int | \nabla \overline{v}_{\alpha_n}|^2 dz = \| \overline{w}_{\alpha_n} \|_{\infty}^{[p(m-1) -2 (m+1)]/2} \int | \nabla \overline{w}_{\alpha_n} |^2 dz \rt 0 \quad \text{as} \quad \alpha \rt \infty.
\]
Then $B_{m+1}(- z_{\alpha_n} \| \overline{w}_{\alpha_n} \|_{\infty}^{(p-2)/2}, \alpha_n \| \overline{w}_{\alpha_n} \|_{\infty}^{(p-2)/2} ) \rt \Omega$ as $\alpha_n \rt \infty$, where $\Omega = \R^{m+1}$ or $\Omega$ is a (possibly affine) half-space in $\R^{m+1}$  and we get the existence of $v$ a solution of
\[
-\Delta v = 0 \quad \text{in} \quad \Omega \quad \text{with} \quad \| v \|_{\infty} =1, \quad v \in \mathcal{D}^{1,2}_0(\Omega),
\]
which contradicts the classical Liouville's theorem. 
\end{proof}

Let $0 \leq \tau_{\alpha} < 1$, see \eqref{maximum value partial henon equation} and \eqref{w and v}, such that 
\[
\max_{z \in B_{m+1}(0,1)} w_{\alpha}(z) = w_{\alpha}(- \tau e_{m+1}).
\]

\begin{lemma}\label{first estimate maximum point}
The product
 \[
 \alpha (1 - \tau_{\alpha}) \quad \text{remains bounded as} \quad \alpha \rt \infty.
 \]
\end{lemma}
\begin{proof}
By contradiction assume that there exists a sequence $(\alpha_n)$ such that
\[
\alpha_n \rt \infty, \quad \alpha_n (1- \tau_{\alpha_n}) \rt \infty \quad \text{as} \quad n \rt \infty.
\]
Set 
\[
\widetilde{w}_{\alpha_n} (z) = \alpha^{-2/(p-2)} w_{\alpha_n} \left( \frac{z}{\alpha_n} - \tau_{\alpha_n} e_{m+1} \right), 
\]
for $z \in \Omega_n : =B_{m+1}(\alpha_n \tau_{\alpha_n} e_{m+1}, \alpha_n)$. Then
\[
\left\{
\begin{array}{l}
 -\Delta \widetilde{w}_{\alpha_n} = h_{\alpha_n}\left( \frac{z}{\alpha_n} - \tau_{\alpha_n} e_{m+1} \right) (\widetilde{w}_{\alpha_n})^{p -1}, \ z \in\Omega_n \ \text{with} \ \widetilde{w}_{\alpha_n} = 0 \ \text{on} \ \partial \Omega_n, \\ \\
 
 0 < C_1 \leq \widetilde{w}_{\alpha_n} (0) = \max_{\Omega_n} \widetilde{w}_{\alpha_n}  \leq C_2 , \quad  \text{as} \quad n \rt \infty \ \text{by \eqref{E1}},\\ \\
 
(\widetilde{w}_{\alpha_n}) \quad \text{is bounded in $\mathcal{D}^{1,2}(\R^{m+1})$ by  \eqref{bound of gradient}}, \\ \\

\Omega_n \rt \R^{m+1} \quad \text{as} \quad n \rt \infty, \\ \\

h_{\alpha_n}\left( \frac{z}{\alpha_n} - \tau_{\alpha_n} e_{m+1} \right) \rt 0 \quad L^{\infty}_{loc}(\R^{m+1}) \quad \text{as} \quad n \rt \infty.
\end{array}
\right.
\]
As a consequence, we obtain $w \in \mathcal{D}^{1,2}(\R^{m+1})$ a bounded positive solution of 
\[
-\Delta w = 0 \quad \text{in} \quad \R^{m+1}, \quad 
\] 
which contradicts the classical Liouville's theorem. 
\end{proof}

\begin{proposition}\label{proposition conv d12}
 We have the convergence
\begin{equation}\label{conv d12}
 \int_{\R^{m+1}_+} | \nabla \widehat{w}_{\alpha} - \nabla w|^2 dz \rt 0 \quad \text{as} \quad \alpha \rt \infty,
\end{equation}
 where
 \[
 \widehat{w}_{\alpha} (z) = \alpha^{-2/(p-2)} w_{\alpha} \left( \frac{z}{\alpha} - e_{m+1}\right), \quad z \in \Omega_{\alpha}: = B_{m+1}(\alpha e_{m+1}, \alpha)
 \]
 and, up to normalizing, $w$ minimizes $m_{1/2, p}$.
 \end{proposition}
\begin{proof}
It follows from \eqref{bound of gradient} that $(\widehat{w}_{\alpha})$ remains bounded in $\mathcal{D}^{1,2}_0(\R^{m+1}_+)$ as $\alpha \rt \infty$. Then there exist $w \in \mathcal{D}^{1,2}_0(\R^{m+1}_+)$ such that $\widehat{w}_{\alpha} \fraco w$ in $\mathcal{D}^{1,2}_0(\R^{m+1}_0)$ as $\alpha \rt \infty$. Observe that
\[
-\Delta \widehat{w}_{\alpha}(z) = h_{\alpha}\left( \frac{z}{\alpha} - e_{m+1} \right) (\widehat{w}_{\alpha})^{p-1}(z) \quad \text{in} \quad \Omega_{\alpha} \quad \text{and} \quad \widehat{w}_{\alpha} = 0 \quad \text{on} \quad \partial \Omega_{\alpha}.
\]
Also observe that
\[
0 \leq h_{\alpha}\left( \frac{z}{\alpha} - e_{m+1} \right) \leq 1 \quad \forall \, z \in \Omega_{\alpha}
\]
and from \eqref{E2}, there exist $C_1, C_2 >0$ such that
\[
C_1 \leq \widehat{w}_{\alpha}(z) \leq C_2 \quad \forall \, z \in \Omega_{\alpha}.
\]
Then, from classical regularity results for second order elliptic equations as in \cite{agmon-douglis-nirenberg} and classical Sobolev imbeddings, we obtain that 
\begin{equation}\label{convergence c1 loc}
\widehat{w}_{\alpha} \rt w \quad \text{in}  \quad C^1_{loc}(\R^{m+1}_+).
\end{equation}
 Now observe that
\[
h_{\alpha}\left( \frac{z}{\alpha} - e_{m+1} \right) = \left. \left( \frac{|z - \alpha e_{m+1}| - (z_{m+1} - \alpha)}{2\alpha} \right)^{\alpha/2} \middle/ \left| \frac{z}{\alpha} - e_{m+1} \right| \right. \forall \, z \in \Omega_{\alpha}.
\]
Then, as we did at \eqref{exponential}, we conclude that $w$ solves
\begin{equation}\label{limit equation *}
 \left\{
\begin{array}{l}
 -\Delta w = e^{- z_{m+1}/2} w^{p-1} \quad \text{in} \quad \R^{m+1}_+,\\
 w> 0 \quad \text{in} \quad \R^{m+1}_+ \quad \text{and} \quad w \in \mathcal{D}^{1,2}_0(\R^{m+1}_+).
\end{array}
 \right.
\end{equation}
Then, from \eqref{limit equation *} and from the definition of $m_{1/2,p}$ we conclude that
\begin{equation}\label{inequality level limit}
\int_{\R^{m+1}_+} | \nabla w|^2 dz  = \int_{\R^{m+1}_+} e^{- z_{m+1}/2} w^{p} dz \geq m_{1/2,p}^{p/(p-2)}.
\end{equation}

With $R> 0$ large with $R < \alpha$ we define
\[
B_{R, \alpha} = \left\{ z; \frac{z}{\alpha} - e_{m+1} \in B_{m+1} (-e_{m+1}, R/\alpha) \cap B_{m+1}(0,1) \right\}.
\]
From \eqref{almost sharp} and with the change of variables $x = \frac{z}{\alpha} - e_{m+1}$ we have
\begin{multline*}
\left( m_{1/2,p}^{p/{p-2}} + o(1) \right) \alpha^{[2(m+1) - p (m-1)]/(p-2)} \geq (S'_{\alpha,p})^{p/(p-2)}\\
 = \int_{B_{m+1}(0,1)} | \nabla w_{\alpha}|^2 dx  = \int_{B_{m+1}(0,1) \cap B_{m+1} (-e_{m+1}, R/\alpha)} | \nabla w_{\alpha}|^2 dx \\
+  \int_{B_{m+1}(0,1) \menos B_{m+1} (-e_{m+1}, R/\alpha)} | \nabla w_{\alpha}|^2 dx \\
= \alpha^{[2(m+1) - p (m-1)]/(p-2)} \left[ \int_{B_{R,\alpha}} \!\!\!\! | \nabla \widehat{w}_{\alpha}|^2 dz + \int_{B_{m+1}(\alpha e_{m+1}, \alpha) \menos B_{R,\alpha}} \!\!\!\!\!\!\!\! | \nabla \widehat{w}_{\alpha}|^2 dz  \right]\\ \geq \alpha^{[2(m+1) - p (m-1)]/(p-2)}\int_{B_{R,\alpha}} | \nabla \widehat{w}_{\alpha}|^2 dz.
\end{multline*}

Then from \eqref{convergence c1 loc} and \eqref{inequality level limit} it follows that
\begin{multline*}
\left( m_{1/2,p}^{p/{p-2}} + o(1) \right) \geq \int_{B_{R,\alpha}} | \nabla \widehat{w}_{\alpha}|^2 dz = \int_{\R^{m+1}_+ \cap B_{m+1}(0, R)} \!\!\!\! | \nabla \widehat{w}_{\alpha}|^2 dz + o(1) \\
=  \int_{\R^{m+1}_+ \cap B_{m+1}(0, R)} | \nabla w|^2 dz + o(1) \geq m_{1/2,p}^{p/(p-2)} + o_{R}(1) + o(1).
\end{multline*}

Hence we obtain
\[
\gd{\int_{B_{m+1}(\alpha e_{m+1}, \alpha) \menos B_{R, \alpha}}  | \nabla \widehat{w}_{\alpha}|^2 dz  = o(1) + o_R(1) \quad \text{when} \quad R< \alpha, \, R \rt \infty,}
\]
\begin{multline*}
\gd{\int_{B_{R, \alpha}}  | \nabla \widehat{w}_{\alpha}|^2 dz  = \int_{\R^{m+1}_{+} \cap B_{m+1}(0,R)} | \nabla w|^2 dz + o(1)} \\
= m_{1/2,p}^{p/(p-2)} + o_{R}(1) + o(1) \quad \text{when} \quad R< \alpha, \, R \rt \infty.
\end{multline*}

Then, from  \eqref{convergence c1 loc} and since $( \widehat{w}_{\alpha})$ is bounded in $\mathcal{D}^{1,2}_0(\R^{m+1}_+)$, it follows that
\begin{multline*}
  \int_{\R^{m+1}_+} | \nabla \widehat{w}_{\alpha} - \nabla w|^2 dz =  \int_{\R^{m+1}_+ \cap B_{m+1}(0, R)} | \nabla \widehat{w}_{\alpha} - \nabla w|^2 dz \\
   +  \int_{\R^{m+1}_+ \menos B_{m+1}(0,R)} | \nabla \widehat{w}_{\alpha} - \nabla w|^2 dz\\
  \leq o(1) + \int_{B_{m+1}(\alpha e_{m+1}, \alpha) \menos B_{R, \alpha}}  | \nabla \widehat{w}_{\alpha} - \nabla w|^2 dz + \int_{\R^{m+1}_+ \menos B_{m+1}(0,R)} | \nabla w|^2 dz\\
  = o(1) + o_{R}(1) - 2 \int_{B_{m+1}(\alpha e_{m+1}, \alpha) \menos B_{R, \alpha}}  \nabla \widehat{w}_{\alpha}  \nabla w dz \leq o(1) + o_{R}(1) \\
  + 2 C \left(  \int_{B_{m+1}(\alpha e_{m+1}, \alpha) \menos B_{R, \alpha}}  |\nabla w|^2 dz\right)^{1/2} = o(1) + o_{R}(1).
\end{multline*}

Hence we conclude that
\[
 \int_{\R^{m+1}_+} | \nabla \widehat{w}_{\alpha} - \nabla w|^2 dz \rt 0 \quad \text{as} \quad \alpha \rt \infty,
\] 
and that $\int_{\R^{m+1}_+} | \nabla w|^2 dz = m_{1/2, p}^{p/(p-2)}$ and so $w$ minimizes $m_{1/2,p}$. 
\end{proof}

\begin{proposition}\label{proposition sharp maximum point}
 There exists $l > 0$ such that
 
\begin{equation}\label{sharp maximum point}
 \alpha (1 - \tau_{\alpha}) \rt l \quad \text{as} \quad \alpha \rt \infty.
\end{equation}
\end{proposition}
\begin{proof}
We have proved that $\widehat{w}_{\alpha} \rt w$ in $C^1_{loc}(\R^{m+1}_+)$ as $\alpha \rt \infty$ and from Lemma \ref{first estimate maximum point} we know that $\alpha (1 -\tau_{\alpha})$ remains bounded as $\alpha \rt \infty$.

Let $z_{\alpha}$ be the maximum point of $\widehat{w}_{\alpha}$. Then
\begin{equation}\label{maximum point limit}
- \tau_{\alpha} e_{m+1} = \frac{z_{\alpha}}{\alpha} - e_{m+1} \quad \text{which implies} \quad z_{\alpha} = \alpha(1 - \tau_{\alpha}) e_{m+1}.
\end{equation}
and we obtain that the maximum point of $\widehat{w}_{\alpha}$ converges to the maximum point of $w$, which is precisely $l e_{m+1}$ for some $l> 0$, which follows from \eqref{limit equation *} and the moving planes technique as in \cite{gidas-ni-nirenberg}. Indeed $w$ is axially symmetric with respect to $\R e_{m+1} \con \R^{m+1}$ and decreasing with respect to $|z'|$. Therefore, from \eqref{maximum point limit} we conclude that
\[
\alpha(1- \tau_{\alpha}) \rt l \quad \text{as} \quad \alpha \rt \infty.
\] 
\end{proof}

\begin{proof}[Proof of Proposition {\rm \ref{proposition asymptotic partial henon reduced}}]
From \eqref{conv d12} we obtain that
\begin{multline*}
m_{1/2, p}^{p/(p-2)} = \int_{\R^{m+1}_+} | \nabla w|^2 dz = \int_{B_{m+1}(\alpha e_{m+1}, \alpha)} | \nabla \widehat{w}_{\alpha}|^2 dz + o(1)\\
= \alpha^{[ p (m-1)- 2(m+1)]/(p-2)} \int_{B_{m+1}(0,1)} | \nabla w_{\alpha}|^2 dx + o(1)\\
 =  \alpha^{[ p (m-1)- 2(m+1)]/(p-2)} (S'_{\alpha,p})^{p/(p-2)}  + o(1).
\end{multline*}
Therefore
\[
 \dfrac{S'_{\alpha,p}}{\alpha^{[2(m+1) - p(m-1)]/p}} = m_{1/2,p} + o(1) \quad \text{as} \quad \alpha \rt \infty. 
\] 
\end{proof}

\medskip

\begin{proof}[Proof of Theorem {\rm \ref{theorem partial henon concentration}}] It follows from Lemma \ref{partial maximum value asymptotic}, Propositions \ref{proposition conv d12} and \ref{proposition sharp maximum point}. 
\end{proof}

\section{Hyperplanes preventing diffusion and proof of Theorem \ref{theorem two point concentration}} \label{section two point concentration}

In this section we consider the problem
\begin{equation}\label{hyper-plane of black holes proof}
\left\{
\begin{array}{l}
 -\Delta_{N} u = |z_N|^{\alpha} |u|^{p-2}u, \qquad z= (z_1, \ldots, z_N) \in B_{N}(0,1),\\
 u = 0 \quad \text{on} \quad \partial B_{N}(0,1).
\end{array}
\right.
\end{equation}
where $p$ and $N$ satisfy the conditions from Theorem \ref{theorem two point concentration}. The procedure to study \eqref{hyper-plane of black holes proof} is quite similar to that from Section \ref{section partial}, but due to its technicality we also include some details here.

By the moving planes technique \cite{gidas-ni-nirenberg} we know that any classical positive solution of \eqref{hyper-plane of black holes proof} is such that $u(z_1, \ldots, z_{N-1}, z_N) = u (|(z_1, \ldots, z_{N-1})|, z_N)$ and that $u(\cdot, z_N)$ decreases with respect to $|(z_1, \ldots, z_{N-1})|$. Therefore, if we look for positive solutions of \eqref{hyper-plane of black holes proof} such that \lbk $u(z_1, \ldots, z_{N-1}, z_N) = u (|(z_1, \ldots, z_{N-1})|, |z_N|)$  we obtain that for any such solution, there exists $0 \leq r < 1$ such that
\[
\gd{\max_{(z_1, \ldots, z_N) \in B_{N}(0,1)} = u(r e_N) = u(-r e_N)}.
\]

Now let $u_{\alpha}$ be a least energy solution among the solutions of \eqref{hyper-plane of black holes proof} that depend only on $|(z_1, \ldots, z_{N-1})|$ and $|z_N|$. Then, by the principle of symmetric criticality \cite{palais}, we characterize such solution as a minimizer of a Rellich quotient among the functions in $H^1_0(B_{N}(0,1))$ invariant by the action of the group
\[
\boldsymbol{\mathfrak{G}}_N = \mathcal{O}(N-1) \times \Z_ 2.
\] 
We can assume that $u_{\alpha} > 0$ in $B_{N}(0,1)$. So arguing as in the previous paragraph, there exists $0 \leq r_{\alpha} < 1$ such that
\[
\mathbf{M}_{\alpha}: = \max_{(z_1, \ldots, z_N) \in B_{N}(0,1)} u_{\alpha}(z_1, \ldots, z_N) = u_{\alpha}(r_{\alpha} e_n) = u_{\alpha}(-r_{\alpha} e_n).
\]

We set
\[
\mathcal{H}_{D,N} : = \{ u  \in H^1_ 0(B_{N}(0,1)); g u = u \; \forall \; g \in \boldsymbol{\mathfrak{G}}_N  \},
\]
the space of functions in $H^1_ 0(B_{N}(0,1))$ that are axially symmetric with \linebreak respect to $\R e_{N} \con \R^{N}$ and symmetric with respect to $x_N$. We also set
\[
\gd{ K_{\alpha,p} : = \inf_{\psi \in H^1_0(B_N(0,1)) \menos \{ 0\}} \dfrac{\int | \nabla \psi|^2 dz}{\left(\int |z|^{\alpha} |\psi|^p dz \right)^{2/p}} }
\]
and
\[
\gd{K'_{\alpha, p} : =  \inf_{w \in\mathcal{H}_{D,N}\menos \{ 0\}} \dfrac{\int | \nabla w|^2 dz}{ \left( \int |z_N|^{\alpha} |w|^p dz\right)^{2/p}} }.
\]

Then, from \cite{smets-willem,pacella}, we have that any minimizer $\psi$ of $K_{\alpha,p}$, up to rotation, is such that $\psi$ is axially symmetric with respect to $\R e_N$. Then, since $|z_N|^{\alpha} \leq |z|^{\alpha}$, we conclude that 
\begin{equation}\label{sobolev constants 1 hyper plane}
K'_{\alpha, p} > K_{\alpha,p} \quad \text{for every}  \quad \alpha > 0.
\end{equation}

We recall that 
\begin{equation}\label{asymptotic henon hyper plane}
 \dfrac{K_{\alpha,p}}{\alpha^{[2N - p(N-2)]/p}} =m_{1,p} + o(1) \quad \text{as} \quad \alpha \rt \infty,
\end{equation}
where
\[
m_{\gamma, p} = \inf \left\{ \int | \nabla w |^2 dz; w \in \mathcal{D}^{1,2}_0(\R^{N}_+), \ \int_{\R^{N}_+} e^{-\gamma z_{N}} |w|^p dz =1\right\},
\]
which is attained for every $\gamma>0$ and $2 < p < \frac{2N}{N-2}$; see \cite[Remark 4.8 and Theorem 2.1]{cao-peng-yan}. In particular, from \eqref{asymptotic henon hyper plane}, there exist $C_1, C_2 >0$ such that
\begin{equation}\label{henon from below and above hyper plane}
C_1 \alpha^{[2N - p(N-2)]/p} \leq K_{\alpha,p} \leq C_2 \alpha^{[2N - p(N-2)]/p} \quad \text{as} \quad \alpha \rt \infty.
\end{equation}

Moreover, the equation
\begin{equation}\label{limit problem henon hyper plane}
-\Delta w = e^{-z_{N}} |w|^{p-2} w \quad \text{in} \quad \R^{N}_+
\end{equation}
is called the limit problem associated to 
\begin{equation}\label{henon hyper plane}
-\Delta u  = |z|^{\alpha} |u|^{p-2}u  \quad \text{in} \quad B_N(0,1), \quad u = 0 \quad \text{on} \quad \partial B_{N}(0,1),
\end{equation}
since after suitable rescaling, we can show that  least energy solutions of \eqref{henon hyper plane} converges to least energy solutions of \eqref{limit problem henon hyper plane} as $\alpha \rt \infty$.

Next we prove that $K'_{\alpha,p}$ may also be controlled as in \eqref{henon from below and above hyper plane}. Indeed we show that the limit problem associated to \eqref{hyper-plane of black holes proof}, for solutions that are axially symmetric with respect to $\R e_{N} \con \R^{N}$ and symmetric with respect to $x_N$, is also \eqref{limit problem henon hyper plane}.

\begin{proposition}\label{proposition asymptotic partial henon reduced hyper plane}
There holds
\begin{equation}\label{asymptotic partial henon reduced hyper plane}
 \dfrac{K'_{\alpha,p}}{\alpha^{[2N - p(N-2)]/p}} = 2^{1 - 2/p}m_{1,p} + o(1) \quad \text{as} \quad \alpha \rt \infty.
\end{equation}
\end{proposition}

We prove some preliminary lemmas in order to go through the proof of Proposition \ref{proposition asymptotic partial henon reduced hyper plane}.

\begin{lemma}
There exist $C_1, C_2$ positive constants such that
\begin{equation}\label{partial from below and above hyper plane}
C_1 \alpha^{[2N - p(N-2)]/p} \leq K'_{\alpha,p} \leq C_2 \alpha^{[2N - p(N-2)]/p} \quad \text{as} \quad \alpha \rt \infty.
\end{equation}

\end{lemma} 
\begin{proof}
Given $\epsilon >0$, choose $u_{\epsilon} \in C_c^{\infty}(\R^{N}_+)$ such that, $u_{\epsilon} \neq 0$, $u_{\epsilon}$ is axially symmetric with respect to $\R e_{N} \con \R^{N}$ and
\[
\dfrac{\int_{\R^{N}_+} |\nabla u_{\epsilon}(s)|^2 ds}{ \left( \int_{\R^{N}_+} e^{- s_{N}} |u_{\epsilon}(s)|^p ds \right)^{2/p}} < m_{1,p} + \epsilon.
\]

Set
\[
\overline{u}_{\epsilon} (z) = u_{\epsilon} (\alpha z', \alpha[ (1- |z'|^2)^{1/2} - |z_{N}|]), \quad z = (z', z_{N}) \in B_{N}(0,1).
\]
Then, it is easy to see that $\overline{u}_{\epsilon} \in\mathcal{H}_{D,N}$ for any $\alpha>0$.

We will perform the change of variables
\begin{equation}\label{change of variables reduced hyper plane}
x = \alpha e_{N} + \alpha z \quad \text{and} \quad s' = x', \, s_{N} =  x_{N} + \alpha (-1 + (1 - \alpha^{-2} |x'|^2)^{1/2}). 
\end{equation}
Then, since $u_{\epsilon}$ has compact support in $\R^{N}_+$, for any $\alpha$ large we get:
\begin{multline*}
\int_{B_{N}(0,1)} \!\!\!\!\!\!\!\! | \nabla \overline{u}_{\epsilon}|^2 dz = 2 \alpha^2 \!\! \int_{B_{N}(0,1), \ z_N < 0}\!\!  \left\{\sum_{i =1}^{N-1} \left[\left|\partial_i u_{\epsilon}(\alpha z', \alpha[ (1- |z'|^2)^{1/2} + z_{N}]) \right. \right. \right. \\
\left. \left. - \frac{z_i}{(1- |z'|^2)^{1/2}}\partial_{N}u_{\epsilon}(\alpha z', \alpha[ (1- |z'|^2)^{1/2} + z_{N}])\right|^2 \right] \\
\left.+ \left| \partial_{N} u_{\epsilon}(\alpha z', \alpha[ (1- |z'|^2)^{1/2} + z_{N}]) \right|^2 \right\}dz \\
= 2 \alpha^{2-N}\int_{B_{N}(\alpha e_{N}, \alpha), \ x_N < \alpha}  \left\{\sum_{i =1}^{N-1} \left[\left|\partial_i u_{\epsilon}(x', (\alpha^2- |x'|^2)^{1/2} + x_N - \alpha) \right. \right. \right. \\
\left. \left. - \frac{\alpha^{-1} x_i}{(1 - \alpha^{-2}|x'|^2)^{1/2}}\partial_{N}u_{\epsilon}(x', (\alpha^2- |x'|^2)^{1/2} + x_N - \alpha)\right|^2 \right] \\
\left. + \left| \partial_{N} u_{\epsilon}(x', (\alpha^2- |x'|^2)^{1/2} + x_N - \alpha) \right|^2 \right\}dx\\
= 2\alpha^{2-N} \int_{\R^{N}_+} \left\{\sum_{i =1}^{N-1} \left[\left|\partial_i u_{\epsilon}(s) - \frac{\alpha^{-1} s_i}{(1 - \alpha^{-2}|s'|^2)^{1/2}}\partial_{N}u_{\epsilon}(s)\right|^2 \right] \right.\\
\left.+ \left| \partial_{N} u_{\epsilon}(s) \right|^2 \right\}ds = 2\alpha^{2-N}\left[ \int_{\R^{N}_+} | \nabla u_{\epsilon}(s)|^2 ds + O(\alpha^{-1})  \right].
\end{multline*}

On the other hand, by the change of variables \eqref{change of variables reduced hyper plane}, we have that for $z \in B_N(0,1)$ with $z_N < 0$ that $0< x_N < \alpha$ and 
\[
|z_N|^{\alpha} = \left| 1 - \dfrac{x_N}{\alpha} \right|^{\alpha} = \left|  \frac{s_N}{\alpha} - (1 - \alpha^{-2}|s'|^2)^{1/2}\right|^\alpha.
\]
Now, if $s \in supp \,u_{\epsilon}$, then
\[
\left| \frac{s_N}{\alpha} - (1 - \alpha^{-2}|s'|^2)^{1/2} \right| = 1 - \frac{s_N}{\alpha} + O(\alpha^{-2})
\]
and
\begin{equation}\label{exponential hyper plane}
|z_N|^{\alpha} = e^{-s_N + O(\alpha^{-1})} + O(\alpha^{-1}).
\end{equation}
Hence
\begin{multline*}
\int_{B_{N}(0,1)} |z_N|^{\alpha} \overline{u}^p_{\epsilon}(z) dz = 2 \int_{B_{N}(0,1), \ z_N <0} |z_N|^{\alpha} \overline{u}^p_{\epsilon}(z) dz \\
=  2 \alpha^{-N} \left[ \int_{\R^{N}_+} e^{-s_N  + O(\alpha^{-1})} u_{\epsilon}^p(s)ds + O(\alpha^{-1})  \right] \\
= 2\alpha^{-N} \left[ \int_{\R^{N}_+} e^{-s_N} u_{\epsilon}^p(s)ds + O(\alpha^{-1})  \right].
\end{multline*}
By the definition of $K'_{\alpha,p}$, we have
\begin{multline*}
K'_{\alpha, p} \leq 2^{1 - 2/p}\alpha^{[2N - p(N-2)]/p}\dfrac{\int_{\R^{N}_+} | \nabla u_{\epsilon}|^2 ds + O(\alpha^{-1})}{\left( \int_{\R^{N}_+} e^{-s_N} u_{\epsilon}^p(s)ds + O(\alpha^{-1}) \right)^{2/p}} \\
= 2^{1 - 2/p}\alpha^{[2N - p(N-2)]/p}\dfrac{\int_{\R^{N}_+} | \nabla u_{\epsilon}|^2 ds }{\left( \int_{\R^{N}_+} e^{-s_N} u_{\epsilon}^p(s)ds \right)^{2/p} + O(\alpha^{-1})} \\
\leq 2^{1 - 2/p}\alpha^{[2N - p(N-2)]/p}(m_{1,p} + \epsilon) + O(\alpha^{-1}).
\end{multline*}
From \eqref{sobolev constants 1 hyper plane}, \eqref{henon from below and above hyper plane} and the last inequality we have that there exist $C_1>0$ such that
\begin{equation}\label{almost sharp hyper plane}
C_1 \leq \dfrac{K'_{\alpha,p}}{\alpha^{[2N - p(N-2)]/p}} \leq 2^{1 - 2/p}m_{1,p} + o(1) \quad \text{as} \quad \alpha \rt \infty.
\end{equation} 
\end{proof}

Let $u_{\alpha} > 0$ be a least energy solution among those which are axially symmetric with respect to $\R e_{N} \con \R^{N}$ and symmetric with respect to $x_N$ solutions of \eqref{hyper-plane of black holes proof}. Then
\[
\int_{B_{N}(0,1)} | \nabla u_{\alpha}|^2 dz = \int_{B_{N}(0,1)} |z_N|^{\alpha} u_{\alpha}^p dz = \left(K'_{\alpha,p}\right)^{p/{p-2}}
\]
and from \eqref{partial from below and above hyper plane}, there exist $C_1, C_2$ positive constants such that
\begin{multline}\label{gradient w hyper plane}
C_1 \alpha^{[2N - p(N-2)]/(p-2)} \leq \int_{B_{N}(0,1)} | \nabla u_{\alpha}|^2 dz = \int_{B_{N}(0,1)} |z_N|^{\alpha} u_{\alpha}^p dz \\
\leq C_2  \alpha^{[2N - p(N-2)]/(p-2)} \ \text{as} \ \alpha \rt \infty.
\end{multline}

Set
\[
\overline{u}_{\alpha}(z) = \alpha^{-2/(p-2)} u_{\alpha} \left( \frac{z}{\alpha}\right), \quad z \in B_{N}(0, \alpha).
\]
Then we have
\[
-\Delta \overline{u}_{\alpha}  =\left| \frac{z_N}{\alpha}\right|^{\alpha} (\overline{u}_{\alpha})^{p-1}, \quad z \in B_{N}(0, \alpha) \quad \text{with} \quad  \overline{w}_{\alpha} = 0 \quad \text{on} \quad \partial B_{N}(0, \alpha)
\]
and
\[
\int_{B_{N}(0, \alpha)} | \nabla \overline{u}_{\alpha}|^2 dz = \alpha^{- [2N - p(N-2)]/(p-2) }\int_{B_{N}(0,1)} | \nabla u_{\alpha}|^2 dz 
\]
and hence
\begin{equation}\label{bound of gradient hyper plane}
C_1 \leq \int_{B_{N}(0, \alpha)} | \nabla \overline{u}_{\alpha}|^2 dz \leq C_2 \quad \text{as} \quad \alpha \rt \infty.
\end{equation}

Then we can proceed as in Section \ref{section partial} to prove the estimate below.

\begin{lemma}\label{partial maximum value asymptotic hyper plane}
 There exist $C_1, C_2$ positive constants such that
 \begin{equation}\label{E1 hyper plane}
 C_1 \leq \max_{z \in B_{N}(0, \alpha)}\overline{u}_{\alpha} (z) \leq C_2 \quad \text{as} \quad \alpha \rt \infty,
 \end{equation}
 that is,  \begin{equation}\label{E2 hyper plane}
 C_1 \alpha^{2/(p-2)}\leq \max_{z \in B_{N}(0, 1)}u_{\alpha} (z) \leq C_2 \alpha^{2/(p-2)} \quad \text{as} \quad \alpha \rt \infty.
 \end{equation}
\end{lemma}
\begin{proof}
It follows exactly as in the proof of Lemma \ref{partial maximum value asymptotic}. 
\end{proof}

Let $0 \leq r_{\alpha} < 1$ such that 
\[
\max_{z \in B_{N}(0,1)} u_{\alpha}(z) = u_{\alpha}(- r_{\alpha} e_{N}) = u_{\alpha}(r_{\alpha}e_N).
\]
We can follow the proof of Lemma \ref{first estimate maximum point} to get the estimate below.
\begin{lemma}\label{first estimate maximum point hyper plane}
The product
 \[
 \alpha (1 - r_{\alpha}) \quad \text{remains bounded as} \quad \alpha \rt \infty.
 \]
\end{lemma}

\begin{proposition}\label{proposition conv d12 hyper plane}
 We have the convergence
\begin{equation}\label{conv d12 hyper plane}
 \int_{\R^{N}_+} | \nabla \widehat{u}_{\alpha} - \nabla u|^2 dz \rt 0 \quad \text{as} \quad \alpha \rt \infty,
\end{equation}
 where
 \[
 \widehat{u}_{\alpha} (z) = \alpha^{-2/(p-2)} u_{\alpha} \left( \frac{z}{\alpha} - e_{N}\right), \quad z \in \Omega_{\alpha}: = \{ z\in B_{N}(\alpha e_{N}, \alpha);  z_N < \alpha\}
 \]
 and, up to normalization, $u$ minimizes $m_{1, p}$.
 \end{proposition}
\begin{proof}
It follows from \eqref{bound of gradient hyper plane} that $(\nabla \widehat{u}_{\alpha})$ remains bounded in $L^2(\R^{m+1}_+)$ as $\alpha \rt \infty$. Observe that
\[
\left\{
\begin{array}{l}
-\Delta \widehat{u}_{\alpha}(z) = \left| \frac{z_N}{\alpha} - e_{N}\right|^{\alpha} (\widehat{u}_{\alpha})^{p-1}(z) \quad \text{in} \quad \Omega_{\alpha},\\
\widehat{w}_{\alpha}(z) = 0 \quad \text{on} \quad z \in \partial \Omega_{\alpha} \quad \text{s.t.} \quad 0 \leq z_N < \alpha,\\
\frac{\partial \widehat{w}}{\partial \nu}(z) = 0 \quad \text{on} \quad z \in \partial \Omega_{\alpha} \quad \text{s.t.} \quad z_N = \alpha.
\end{array}
\right.
\]
Also observe that
\[
0 \leq \left| \frac{z_N}{\alpha} - e_{N}\right|^{\alpha} \leq 1 \quad \forall \, z \in \Omega_{\alpha}
\]
and from \eqref{E2 hyper plane}, there exist $C_1, C_2 >0$ such that
\[
C_1 \leq \widehat{u}_{\alpha}(z) \leq C_2 \quad \forall \, z \in \Omega_{\alpha}.
\]
Then, from classical regularity results for second order elliptic equations as in \cite{agmon-douglis-nirenberg} and classical Sobolev imbeddings, we obtain that there exists $w \in \mathcal{D}^{1,2}_0(\R^{N}_+)$
\begin{equation}\label{convergence c1 loc hyper plane}
\widehat{u}_{\alpha} \rt u \quad \text{in}  \quad C^1_{loc}(\R^{N}_+).
\end{equation}
Then we conclude that $w$ solves
\begin{equation}\label{limit equation * hyper plane}
 \left\{
\begin{array}{l}
 -\Delta u = e^{- z_{N}} u^{p-1} \quad \text{in} \quad \R^{N}_+,\\
 u> 0 \quad \text{in} \quad \R^{N}_+ \quad \text{and} \quad u \in \mathcal{D}^{1,2}_0(\R^{N}_+).
\end{array}
 \right.
\end{equation}
Then, from \eqref{limit equation * hyper plane} and from the definition of $m_{1,p}$ we conclude that
\begin{equation}\label{inequality level limit hyper plane}
\int_{\R^{N}_+} | \nabla u|^2 dz  = \int_{\R^{N}_+} e^{- z_{N}} u^{p} dz \geq m_{1,p}^{p/(p-2)}.
\end{equation}

With $R> 0$ large with $R < \alpha$ we define
\[
B_{R, \alpha} = \left\{ z; \frac{z}{\alpha} - e_{N} \in B_{N} (-e_{N}, R/\alpha) \cap B_{N}(0,1) \right\}.
\]
From \eqref{almost sharp hyper plane} and with the change of variables $x = \frac{z}{\alpha} - e_{N}$ we have
\begin{multline*}
\left( 2\, m_{1,p}^{p/{p-2}} + o(1) \right) \alpha^{[2(m+1) - p (m-1)]/(p-2)} \geq (K'_{\alpha,p})^{p/(p-2)} \\ = 2 \int_{B_{N}(0,1), \ z_N<0} | \nabla u_{\alpha}|^2 dx  = 2 \int_{B_{N}(0,1) \cap B_{N} (-e_{N}, R/\alpha)} | \nabla u_{\alpha}|^2 dx \\
+  2 \int_{\left(B_{N}(0,1),\ z_N <0\right) \menos B_{m+1} (-e_{m+1}, R/\alpha)} | \nabla u_{\alpha}|^2 dx \\= 2 \alpha^{[2N - p (N-2)]/(p-2)} \left[ \int_{B_{R,\alpha}} | \nabla \widehat{u}_{\alpha}|^2 dz + \int_{\Omega_{\alpha} \menos B_{R,\alpha}} | \nabla \widehat{u}_{\alpha}|^2 dz  \right]\\
 \geq 2 \alpha^{[2N - p (N-2)]/(p-2)}\int_{B_{R,\alpha}} | \nabla \widehat{u}_{\alpha}|^2 dz.
\end{multline*}

Then from \eqref{convergence c1 loc hyper plane} and \eqref{inequality level limit hyper plane} it follows that
\begin{multline*}
\left( 2\,m_{1,p}^{p/{p-2}} + o(1) \right) \geq 2 \int_{B_{R,\alpha}} | \nabla \widehat{u}_{\alpha}|^2 dz = 2 \int_{\R^{N}_+ \cap B_{N}(0, R)} | \nabla \widehat{u}_{\alpha}|^2 dz + o(1) \\
=  2 \int_{\R^{N}_+ \cap B_{N}(0, R)} | \nabla u|^2 dz + o(1) \geq 2 \, m_{1,p}^{p/(p-2)} + o_{R}(1) + o(1).
\end{multline*}

Hence we obtain
\[
\gd{\int_{\Omega_{\alpha}\menos B_{R, \alpha}}  | \nabla \widehat{u}_{\alpha}|^2 dz  = o(1) + o_R(1) \quad \text{when} \quad R< \alpha, \, R \rt \infty,}
\]
\[
\gd{\int_{B_{R, \alpha}}  | \nabla \widehat{u}_{\alpha}|^2 dz  = \int_{\R^{N}_{+} \cap B_{N}(0,R)} | \nabla u|^2 dz + o(1) = m_{1,p}^{p/(p-2)} + o_{R}(1) + o(1) }
\]
when $R< \alpha, \, R \rt \infty$.

Then, from  \eqref{convergence c1 loc hyper plane} and since $(\nabla \widehat{u}_{\alpha})$ is bounded in $L^2(\R^{N}_+)$, it follows that
\begin{multline*}
  \int_{\R^{N}_+} | \nabla \widehat{u}_{\alpha} - \nabla u|^2 dz =  \int_{\R^{N}_+ \cap B_{N}(0, R)} \!\!\!\! \!\!\!\! \!\!\!\!| \nabla \widehat{u}_{\alpha} - \nabla u|^2 dz +  \int_{\R^{N}_+ \menos B_{N}(0,R)} \!\!\!\! \!\!\!\!\!\!\!\! | \nabla \widehat{u}_{\alpha} - \nabla u|^2 dz\\
  \leq o(1) + \int_{\Omega_{\alpha}\menos B_{R, \alpha}}  | \nabla \widehat{u}_{\alpha} - \nabla u|^2 dz + \int_{\R^{N}_+ \menos B_{N}(0,R)} | \nabla u|^2 dz\\
  = o(1) + o_{R}(1) - 2 \int_{\Omega_{\alpha} \menos B_{R, \alpha}}  \nabla \widehat{u}_{\alpha}  \nabla u dz \leq o(1) + o_{R}(1) \\
  + 2 C \left(  \int_{\Omega_{\alpha} \menos B_{R, \alpha}}  |\nabla u|^2 dz\right)^{1/2}  = o(1) + o_{R}(1).
\end{multline*}

Hence we conclude that
\[
 \int_{\R^{N}_+} | \nabla \widehat{u}_{\alpha} - \nabla u|^2 dz \rt 0 \quad \text{as} \quad \alpha \rt \infty,
\] 
and that $\int_{\R^{N}_+} | \nabla u|^2 dz = m_{1, p}^{p/(p-2)}$ and so $w$ minimizes $m_{1,p}$. 
\end{proof}

\begin{proposition}\label{proposition sharp maximum point hyper plane}
 There exists $l > 0$ such that
\[
 \alpha (1 - \tau_{\alpha}) \rt l \quad \text{as} \quad \alpha \rt \infty.
\]
\end{proposition}
\begin{proof}
Exactly as the proof of Proposition \ref{proposition sharp maximum point}. 
\end{proof}

\begin{proof}[Proof of Proposition {\rm \ref{proposition asymptotic partial henon reduced hyper plane}}]
From \eqref{conv d12 hyper plane} we obtain that
\begin{multline*}
m_{1, p}^{p/(p-2)} = \int_{\R^{N}_+} | \nabla u|^2 dz = \int_{\Omega_{\alpha}} | \nabla \widehat{u}_{\alpha}|^2 dz + o(1)\\
=  \alpha^{[ p (N-2)- 2N]/(p-2)} \int_{B_{N}(0,1), \ z_N <0} | \nabla u_{\alpha}|^2 dx + o(1) \\
=  \frac{1}{2}\alpha^{[ p (N-2)- 2N]/(p-2)} (K'_{\alpha,p})^{p/(p-2)}  + o(1).
\end{multline*}
Therefore
\[
 \dfrac{K'_{\alpha,p}}{\alpha^{[2N - p(N-2)]/p}} = 2^{1 -2/p} m_{1,p} + o(1) \quad \text{as} \quad \alpha \rt \infty.
\] 
\end{proof}

\medskip

\begin{proof}[Proof of Theorem {\rm \ref{theorem two point concentration}}] It follows from Lemma \ref{partial maximum value asymptotic hyper plane}, Propositions \ref{proposition conv d12 hyper plane} and \ref{proposition sharp maximum point hyper plane}. 
\end{proof}

\end{document}